\documentclass{article}
\usepackage[english]{babel}

\usepackage[letterpaper,top=2cm,bottom=2cm,left=3cm,right=3cm,marginparwidth=1.75cm]{geometry}

\usepackage{amsmath}
\usepackage{graphicx}
\usepackage[colorlinks=true, allcolors=blue]{hyperref}
\usepackage{amssymb}
\usepackage{amsthm}

\DeclareMathOperator*{\argmin}{arg\,min}

\newtheorem{theorem}{Theorem}
\newtheorem{lemma}[theorem]{Lemma}
\newtheorem{proposition}[theorem]{Proposition}
\newtheorem{corollary}[theorem]{Corollary}
\theoremstyle{definition}
\newtheorem{definition}[theorem]{Definition}

\newtheorem{remark}{Remark}
\newtheorem{conjecture}[theorem]{Conjecture}

\newcommand{\im}{\mathrm{im}\,}
\newcommand{\MT}{\mathrm{MT}}
\newcommand{\LCA}{\mathrm{LCA}}
\newcommand{\LMT}{\mathrm{LMT}}
\newcommand{\BC}{\mathcal{B}}

\renewcommand{\epsilon}{\varepsilon}

\title{Intrinsic Bottleneck Distance for Merge Trees}
\author{David Beers\thanks{Department of Mathematics, University of California, Los Angeles}\; and Gillian Grindstaff\thanks{Mathematical Institute, University of Oxford}}

\begin{document}
\maketitle

\begin{abstract}
Merge trees are a topological descriptor of a filtered space that enriches the degree zero barcode with its merge structure. The space of merge trees comes equipped with an interleaving distance $d_I$, which prompts the natural question: is the interleaving distance between two merge trees equal to the bottleneck distance between their corresponding barcodes? As the map from merge trees to barcodes is not injective, the answer as posed is no, but (as conjectured in Gasparovic et al.) we prove that it is true for the \emph{intrinsic} metrics $\widehat{d}_I$ and $\widehat{d}_B$ realized by infinitesimal path length in merge tree space. This implies that they have the same induced length space, and that in cases where a path is known, the bottleneck distance (which can be computed quickly) can be substituted for the interleaving distance (in general, NP-hard to approximate).
\end{abstract}

\section{Introduction}

One central topic in applied topology is the multiscale study of shape. Beginning with a family of nested topological spaces $\{X_t\}_{t\in \mathbb{R}}$, the \emph{dimension-0 persistence barcode}, or simply the \emph{barcode}, captures the path components of the family as $t$ evolves. The barcode is a multiset of intervals $[t_b,t_d)$ on which individual components ``persist": the \emph{birth} value $t_b$ at which the component first appears, and the \emph{death} value $t_d$ at which it merges with an older component \cite{zomorodian2004computing}. From this summary, the number of path components in $X_t$ can be recovered as the number of intervals containing $t$.

Crucially, when merging of path components occurs, the barcode marks the death of the later-born component, but retains no information about the component it has merged with. When this information is retained, the components of $\{X_t\}$ can be summarized by a sharper descriptor called a \emph{merge tree}. More precisely defined in Section \ref{sec:background}, a merge tree is a tree-like shape obtained from the barcode assigned to $\{X_t\}_{t\in\mathbb{R}}$ by attaching intervals corresponding to path components at the moment they are merged. The interplay between dimension-0 persistence barcodes and merge trees has been explored extensively \cite{curry2018fiber,TreesBarcodes,curry2024trees}, and used in applications such as neuronal morphology \cite{TMD, li2017metrics, kanari2019objective, beers2023barcodes}.

Barcodes come equipped with a metric\footnote{Technically speaking, the bottleneck distance is an extended pseudometric. On the kinds of barcodes that appear in this paper, those in which each interval has a closed left endpoint and open right endpoint and exactly one (right) endpoint among all intervals is infinite, it is a metric.} $d_B$, called the \emph{bottleneck distance}. Let $(\BC, d_B)$ denote the space of barcodes equipped with the bottleneck distance. The consistency of barcodes in analyzing real data sets depends on foundational stability results; for certain domains of data (including point clouds in $\mathbb{R}^n$ under Gromov-Hausdorff distance and height maps under $\ell_\infty$ distance), distance-based maps to $\BC$ are Lipschitz \cite{CohenSteiner}. Beyond stability, the geometry itself is complex and not fully understood - among other properties, $\BC$ is infinite-dimensional and non-Riemannian, with multiple Fr\'{e}chet means arising from regions of positive curvature \cite{Che2024}. Nevertheless, the distance between two points can be computed efficiently \cite{kerber2017geometry}; it is largely through this metric, and its properties, that statistics can be built on persistence barcodes \cite{FrechetPDs}. Canonical algebraic and mass-transport distances are known to coincide between 1-parameter persistence modules and their barcodes \cite{bauer2014induced}, a key result enabling quick computation of these statistics.

A natural variation of the space of barcodes $\BC$ is the space of all merge trees (referred to throughout as $\MT$) with \emph{interleaving} distance $d_I$. Pointwise, $\MT$ differs from $\BC$ only by the additional combinatorial data of which components merge, so that the forgetful map detaching branches from trees following the elder rule (\cite{cai2021elder, curry2018fiber, edelsbrunner2010computational}, see also Figure \ref{fig:elderrule}) recovers the original barcode.
Through this forgetful map we may define the bottleneck distance $d_B$ on merge trees as well. We have good reason to hope for tractable geometry: in addition to its natural correspondence with the space of barcodes,
many metric spaces of trees have been defined and studied in other applications, such as phylogenetics and data structures. Although phylogenetic analysis focuses on labeled, fixed taxa, metrics based on Euclidean \cite{BHV}, tropical \cite{trop} and probabilistic metrics \cite{wald} can induce a geometry on unlabeled trees as quotients of labeled tree space (as outlined in  \cite{feragen_nye_2020}). Additionally, when one restricts to rooted trees, the \emph{cophenetic} metric uses the height of merge events/least common ancestor to encode point pairs in a correlation matrix, and then takes an $L^p$ norm on the matrices. Munch and Stefanou \cite{munch2019} have observed that when $p = \infty$, this metric is equal to a variant of the interleaving distance for labeled merge trees. Gasparovic et al. \cite{gasparovic2025intrinsic} show that the \emph{intrinsic} interleaving distance $\widehat{d}_I$, as defined by infinitesimal path length in $\MT$, coincides with $d_I$ on the space of merge trees (stated here as Theorem \ref{thm:dI}). That is, every interleaving distance on merge trees can be realized by a geodesic. This is critical for defining Fr\'{e}chet means and further statistics.

 Morozov et al. \cite{morozov2013interleaving} established that $d_I$ dominates $d_B$ in the merge-tree setting, meaning $d_B \leq d_I$. Critically, strict equality does not hold, as $d_B$ is a pseudometric on merge trees that cannot distinguish between different trees associated to the same barcode (e.g. Fig~\ref{fig:mergetreepath}), the multiplicity and combinatorics of which vary \cite{curry2024trees,beers2025fiber}. However, each fiber of the forgetful map from $\MT$ to $\BC$ is a finite isolated set of points in the interleaving metric topology \cite{curry2018fiber}, which implies that the intrinsic bottleneck distance $\widehat{d}_B:\MT\times\MT \to \mathbb{R}_{\geq 0}$ using paths that are continuous (with respect to $d_I$) may be positive where the $d_B$ pseudometric is 0.
In their conclusion, Gasparovic et al. mention (informally) a conjecture: that on the space $\MT$, the \emph{intrinsic} versions of bottleneck and interleaving distances actually fully agree. 

In Theorem \ref{thm:conjecture} we prove this conjecture:
\[\widehat{d}_B = \widehat{d}_I\,,\]
which means that as length spaces, there is an isometry $(\MT, \widehat{d}_B) \cong (\MT, d_I)$.
We also establish, in Section \ref{sec:geodesics}, that geodesics realizing these intrinsic distances exist: for any two merge trees there is a path between them whose $d_B$-length and $d_I$-length both equal their interleaving distance (Theorem \ref{thm:geodesics}).  Together these results extend some of the computational advantages of bottleneck distance from the space of barcodes to merge trees with known geodesics, or general continuous paths of merge trees, and elucidate the relationship between the two geometries. We note that paths of merge trees have been recently used to model time-varying terrain maps \cite{beurskens2025ordered, yan2022geometry} - if the time resolution is fine enough, our result gives a strong justification for using bottleneck distance to compute path length. 

\section{Background}\label{sec:background}

\subsection{Merge Trees}

A \emph{tree} is a finite acyclic graph. A \emph{geometric tree} is any topological space $X$ obtained from a tree $G$ by viewing each edge of $G$ as a copy of the unit interval $[0,1]$ and identifying endpoints corresponding to the same vertex of $G$. Branch points (resp. leaves) in $X$ are the points in $X$ corresponding to branch points (resp. leaves) under this identification.

\begin{definition}
\label{def:mt}
    A \emph{merge tree} $(T,f)$ is a pair consisting of
    \begin{enumerate}
        \item A topological space $T = X\sqcup[0,1)/\sim$, where $X$ is a geometric tree and $\sim$ is the relation $x_0\sim 0$, for some particular $x_0\in X$ (called the \emph{root}).
        \item A continuous function $f: T \to \mathbb{R}$ such that:
            \begin{itemize}
                \item $f$ is strictly increasing on its restriction to the copy of $[0,1)$ in the construction $T = X \sqcup[0,1)/\sim$ and $f(y) \to \infty$ as $y \to 1$ on this interval; and
                \item $f$ is strictly increasing on any injective path $\gamma: [0,1] \to X$ with $\gamma(1) = x_0$.
            \end{itemize}
    \end{enumerate}
    For $u,v \in T$ we write $u \preceq v$ if there is a path from $u$ to $v$ that is strictly increasing in $f$, or if $u=v$. In this case we say $v$ is an ancestor of $u$ or $u$ is a descendant of $v$. The image of a leaf of $X$ in $T$ is called a leaf of $T$, provided it is not identified to $0\in[0,1)$. The image of a branch point, or the image of $0$ provided it is not identified with a leaf, is called a branch point of $T$. If $S$ is the set of leaves and branch points of $T$, then the path components of $T-S$ are the edges of $T$.

    The least common ancestor of two leaves $x,y\in T$, denoted $\LCA(x,y)$, is defined to be the unique point $z\in T$, such that $x,y\preceq z$, and if $x,y \preceq z'$, then $z \preceq z'$. It can be shown that if $x$ and $y$ are distinct leaves of $T$, then $\LCA(x,y)$ is a branch point. Similarly, if $V$ is a finite set of points in $T$, then $\LCA(V)$ is defined as the unique point $z \in T$ such that $x \preceq z$ for all $x \in V$, but for any $z'$ such that $x \preceq z'$ for all $x \in V$, we have $z \preceq z'$.

    Two merge trees $(T,f)$ and $(T',f')$ are called \emph{isomorphic} if there is a homeomorphism $\phi: T \to T'$ such that $f = f'\circ \phi$.
\end{definition}

Since, by definition, a geometric tree has only finitely many leaves, branch points, and edges, the same is true for merge trees.

An example merge tree is illustrated in Figure \ref{fig:mergetreeex}.

\begin{figure}[htbp]
    \centering
    \includegraphics[height=2.5in]{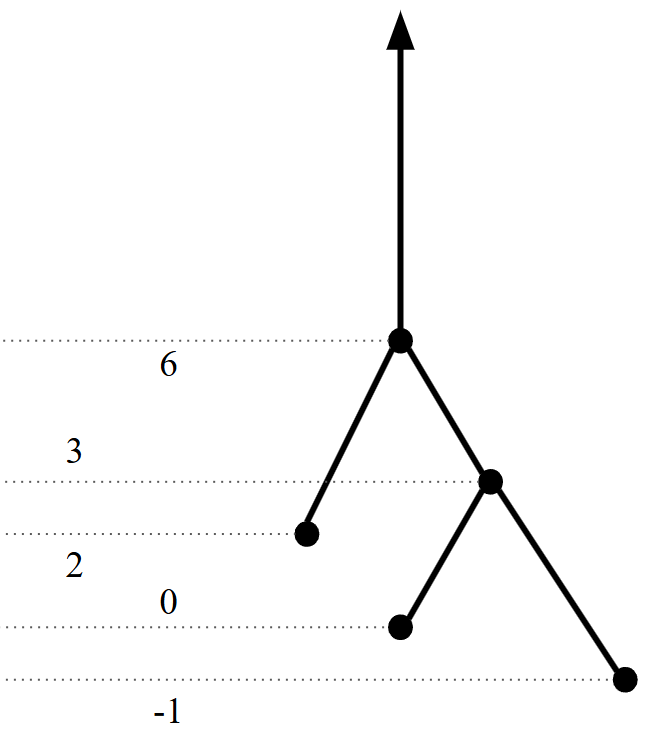}
    \caption{A merge tree plotted by height $f$, with $f$-values indicated at vertices.}
    \label{fig:mergetreeex}
\end{figure}

\begin{remark}
    Sometimes in the literature, what we call in this paper a merge tree is called a \emph{cellular} merge tree (e.g. in \cite{curry2018fiber}) to distinguish from more general constructions considered in \cite{morozov2013interleaving}, for example. We say \emph{merge trees} in this paper instead of \emph{cellular merge trees} in order to remain consistent with the terminology used in \cite{gasparovic2025intrinsic}.
\end{remark}

Given $\epsilon > 0$ and $x\in T$, one can show there is a path $\gamma:[0,1] \to T$ such that $\gamma(0) = x$, $f(\gamma(1)) = f(x) + \epsilon$, $f\circ \gamma$ strictly increasing, and $\gamma$ is unique up to reparameterization. As such, we may define $i^\epsilon(x)$ to be the point $\gamma(1)$ in $T$. We remark here that $i^\epsilon \circ i^\delta = i^{\epsilon + \delta}$.

\begin{definition}
    Given merge trees $(T, f)$ and $(T',f')$, an \emph{$\epsilon$-interleaving} between $(T,f)$ and $(T',f')$ is a pair of continuous maps, $\alpha:T\to T'$, $\beta:T' \to T$ satisfying that for all $x\in T$ and all $x'\in T'$
    \begin{align*}
        f'\circ \alpha(x) = f(x) + \epsilon & \qquad f \circ \beta(x') = f'(x') + \epsilon \\
        \beta \circ \alpha(x) = i^{2\epsilon}(x) & \qquad \alpha\circ\beta(x') = i^{2\epsilon}(x').
    \end{align*}
    The \emph{interleaving distance} between $(T,f)$ and $(T',f')$, denoted $d_I\big((T,f),(T',f')\big)$, is the infimum of values $\epsilon$ such that $(T,f)$ and $(T',f')$ are $\epsilon$-interleaved.
\end{definition}

As remarked in \cite{morozov2013interleaving} it follows from the definition of an interleaving $(\alpha,\beta)$ that
\begin{equation*}
    i^\epsilon \circ\alpha = \alpha \circ i^\epsilon \qquad i^\epsilon \circ\beta = \beta \circ i^\epsilon \qquad \text{for all }\epsilon.
\end{equation*}
The following is known (\cite[Theorem 9]{beurskens2025locally}, see also \cite{pegoraro2025graphmatching}); for completeness, we give a short self-contained proof via compactness that produces the interleaving pair $\alpha, \beta$ directly. As we shall show, Lemma \ref{lem:infint} can be used to show that the interleaving distance is a true metric (Lemma \ref{lem:intmetric}), a result previously shown by other means \cite{cardona2022universal}. Other than its application in the proof of Lemma \ref{lem:intmetric}, we only use Lemma \ref{lem:infint} in Section \ref{sec:geodesics}.

\begin{lemma}
    \label{lem:infint}
    Suppose $d_I\big((T,f),(T',f')\big) = \epsilon$. Then $(T,f)$ and $(T',f')$ are $\epsilon$-interleaved. In other words, the infimum in the definition of $d_I$ is actually a minimum, and maps satisfying the $\epsilon$-interleaving conditions can be specified.
\end{lemma}

\begin{proof}
    If $d_I\big((T,f),(T',f')\big) = \epsilon$, then there exists a sequence $\{(\alpha_k,\beta_k)\}_{k\in \mathbb{N}}$ such that $\alpha_k:T\to T', \beta_k:T'\to T$ define an $\epsilon_k$ interleaving and $\epsilon_k \to \epsilon$. By potentially taking a subsequence we may assume the values $\epsilon_k$ are weakly decreasing. We let $E = \sup\epsilon_k$.

    Let $l_1,\ldots,l_n$ denote the leaves of $T$ and $l'_1,\ldots,l'_m$ denote the leaves of $T'$. Let
    \begin{equation*}
        A = \max\big(\max_{1\leq i\leq n} f(l_i), \max_{1 \leq i\leq m}f'(l'_i)\big) \qquad X = f^{-1}(-\infty,A+E] \qquad X' = (f')^{-1}(-\infty, A+E].
    \end{equation*}
    Then $X$ and $X'$ are compact. We have a sequence indexed by $k$ given by
    \begin{equation*}
        \big(\alpha_k(l_1), \ldots\alpha_k(l_n), \beta_k(l'_1),\ldots,\beta_k(l'_m)\big) \in (X')^n\times X^m.
    \end{equation*}
    Being a sequence in a compact metric space, it has a convergent subsequence. By taking such a convergent subsequence, we may assume the above sequence converges. We refer to the element this sequence converges to as
    \begin{equation*}
         \big(\alpha(l_1), \ldots\alpha(l_n), \beta(l'_1),\ldots,\beta(l'_m)\big).
    \end{equation*}
    It follows that $f'(\alpha(l_i)) = f(l_i) + \epsilon$, $f(\beta(l'_i)) = f'(l'_i) + \epsilon$. This implies $f'(\alpha(l_i)) \leq f'(\alpha_k(l_i))$ for all $k$. For any $x \in T'$, one may find a neighborhood $U_x$ such that if $y \in U_x$ and $f'(x) \leq f'(y)$, then $x \preceq y$. Consequently, by taking a subsequence, we may assume $\alpha(l_i) \preceq \alpha_k(l_i)$ for all $i$ and $k$. By taking another subsequence, a similar argument shows we may assume $\beta(l'_i) \preceq \beta_k(l'_i)$ for all $i$ and $k$.

    For all $x\in T$, we define $\alpha(x) = \lim_{k \to \infty} \alpha_k(x)$, noting that this agrees with our previous definition of $\alpha(l_i)$. Indeed this limit exists since if $x = i^t(l_i)$ we observe, using continuity of $i^t$, that
    \begin{equation}
        \label{eqn:alphadef}
        \alpha(x) = \lim_{k \to \infty} \alpha_k(x) =\lim_{k\to \infty} \alpha_k i^t(l_i) = \lim_{k\to \infty} i^t \alpha_k(l_i) = i^t\lim_{k\to \infty} \alpha_k(l_i) = i^t\alpha(l_i).
    \end{equation}
    If $\gamma_i:[0,1) \to T$ is a path with $\gamma_i(0) = l_i$ along which $f$ is increasing to infinity, the above equation shows that $\alpha$ restricted to the image of $\gamma_i$ is continuous, which implies that $\alpha$ is continuous.
    Similarly we may define $\beta:T' \to T$, and observe that $\beta$ is continuous.
    We have
    \begin{align*}
        i^t \alpha(x) = i^t\lim_{k \to \infty}\alpha_k(x) = \lim_{k \to \infty}i^{t}\alpha_k(x) = \lim_{k \to \infty} \alpha_ki^t(x) = \alpha i^t(x)\\
        f' \alpha(x) = f'\lim_{k \to \infty}\alpha_k(x) = \lim_{k \to \infty}f'\alpha_k(x) = \lim_{k \to \infty} f(x) + \epsilon_k = f(x) + \epsilon.
    \end{align*}
    Similarly $i^t\beta(x) = \beta i^t(x)$ and $f\beta(x) = f'(x) + \epsilon$ for all $x \in T'$.
    Let $\delta_k = \epsilon_k - \epsilon$. Since $\alpha(l_i) \preceq \alpha_k(l_i)$, Equation (\ref{eqn:alphadef}) implies that $\alpha(x) \preceq \alpha_k(x)$ for all $x\in T$. Consequently we have $\alpha_k(x) = i^{\delta_k}\alpha(x)$. Similarly, for $x\in T'$, $\beta_k(x) = i^{\delta_k}\beta(x)$. Therefore, if $j \geq k$,
    \begin{equation*}
        \alpha_j\beta_k(x) = i^{\delta_j}\alpha \beta_k(x) = i^{\delta_j - \delta_k}\alpha_k\beta_k(x) = i^{\delta_j - \delta_k}i^{2\epsilon_k}(x) = i^{\epsilon_j +\epsilon_k}(x).
    \end{equation*}
    If instead $k \geq j$,
    \begin{equation*}
        \alpha_j\beta_k(x) = \alpha_ji^{\delta_k}\beta(x) = \alpha_ji^{\delta_k-\delta_j}\beta_j(x)= i^{\delta_k-\delta_j}\alpha_j\beta_j(x) = i^{\delta_k - \delta_j}i^{2\epsilon_j}(x) = i^{\epsilon_j + \epsilon_k}(x),
    \end{equation*}
    So $\alpha_j \circ \beta_k = i^{\epsilon_j + \epsilon_k}$. Therefore,
    \begin{equation*}
        \alpha\beta(x) = \lim_{j\to\infty}\alpha_j \lim_{k\to \infty} \beta_k(x) = \lim_{j\to \infty}\lim_{k \to \infty} \alpha_j\beta_k(x) = \lim_{j\to \infty}\lim_{k \to \infty} i^{\epsilon_j + \epsilon_k}(x) = i^{2\epsilon}(x).
    \end{equation*}
    So $\alpha \circ \beta = i^{2\epsilon}$. Similarly it is shown that $\beta \circ \alpha = i^{2\epsilon}$, completing the proof.
\end{proof}

The result below is proven in \cite{cardona2022universal} (specifically, it follows immediately from Theorem 1.4 and Lemma 4.4 therein). Since it follows without much difficulty from Lemma \ref{lem:infint} we provide an alternate proof here.

\begin{lemma}
    \label{lem:intmetric}
    The interleaving distance is a metric on merge trees (up to isomorphism).
\end{lemma}

\begin{proof}
    From \cite[Lemma 1]{morozov2013interleaving} we know that $d_I$ satisfies the triangle inequality and symmetry; to verify that it is a metric, it remains to show that $d_I\big((T,f),(T',f')\big)<\infty$ and that $(T,f)=(T',f')$ if $d_I\big((T,f),(T',f')\big)= 0$.

    To handle the first issue, let $(T,f)$, $(T',f')$ be merge trees. We let $e$ (resp. $e'$) be the root edge, i.e. the unique edge of $T$ (resp. $T'$) which takes arbitrarily large $f$-values. We may define an interleaving between the two merge trees sending each point to either $e$ or $e'$ as follows. Let $r\in T, r'\in T'$ be the root nodes, and let $A = \max\{f(r),f'(r')\}$. By the merge tree definition, each $A^*>A$ corresponds to a unique value $x\in T$ in the root edge with $f(x) = A^*$, and similarly for $x'\in T'$. Now let $B=\inf(f,f')$ and define $\alpha(x):= (f')^{-1}(f(x) + A - B).$ Since $f(x)+A-B > A$, the image of $\alpha$ is contained in the root edge of $T'$, and so $(f')^{-1}(f(x) + A - B)$ is uniquely defined on $T$. Define $\beta$ analogously. Observe that $A$ and $B$ are finite, so that the $A-B$ interleaving gives a finite upper bound on $d_I(T,T')$.

    To handle the second issue, we suppose $d_I\big((T,f),(T',f')\big) = 0$. Lemma \ref{lem:infint} implies that there exists $\alpha:T\to T'$ and $\beta:T' \to T$ which form a $0$-interleaving. Hence $\alpha$ defines a merge tree isomorphism from $(T,f)$ to $(T',f')$ with inverse $\beta$.
\end{proof}

\begin{definition}
    We define $\MT$ to be the metric space of isomorphism classes of merge trees equipped with the interleaving distance. We define $\MT_n$ to be the metric subspace consisting of $\MT$ of merge trees with $n$ leaves or less, also equipped with the interleaving distance.
\end{definition}

\subsection{The Bottleneck Distance}

Given a merge tree $(T,f)$, there is a procedure called the elder rule which converts $(T,f)$ into a multiset of intervals. Suppose that $l_0,l_1,\ldots,l_n$ are the leaves of $T$, indexed such that if $f(l_i) < f(l_j)$ then $i<j$, and otherwise indexed arbitrarily. For $1 \leq i \leq n$, define $b_i$ to be the least ancestor of $l_i$ such that $b_i$ is also an ancestor of $l_j$ for some $j < i$. In the next section we prove $b_i$ is well defined, a branch point, and uniquely determined by this definition, although this is well known. One can show that moreover every branch point in $(T,f)$ is a point $b_i$ for some $i$. The elder rule defines
\begin{equation*}
    B(T,f) = \big[f(l_0),\infty \big) \sqcup \bigsqcup_{i = 1}^n \big[f(l_i), f(b_i)\big)\,.
\end{equation*}
The multiset of intervals $B(T,f)$ is known to be unaffected by changing the order of leaves with identical $f$-values, see Remark \ref{rmk:elderruleph} for further details. We illustrate the computation of $B(T,f)$ from a merge tree in Figure \ref{fig:elderrule}.

\begin{figure}[htbp]
    \centering
    \includegraphics[height=2.5in]{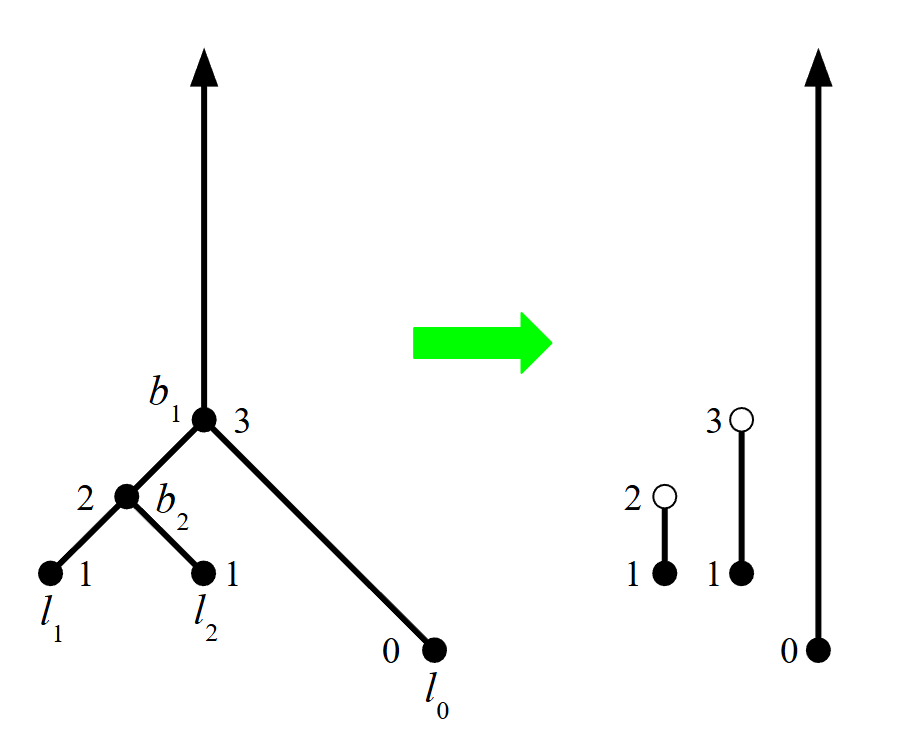}
    \caption{(Left) A merge tree $(T,f)$ with vertices labeled both as either a leaf $l_i$ or a branch point $b_i$ and with their $f$-value. (Right) The barcode $B(T,f)$ obtained from the merge tree $(T,f)$.}
    \label{fig:elderrule}
\end{figure}

Due to this construction, we make the following definition.

\begin{definition}
    A multiset of intervals in the real line is called a \emph{barcode}. We denote by $\BC$ the space of all barcodes.
\end{definition}

\begin{remark}
    \label{rmk:elderruleph}
    One reason the elder rule was originally precisely formulated is that the elder rule applied to $(T,f)$ computes the degree zero sublevel set persistent homology of the function $f$. Since the definition of persistent homology is rather algebraic and not strictly necessary for this paper, we refer the reader to \cite{cai2021elder,curry2018fiber,edelsbrunner2010computational} for the details of the connection between the elder rule and persistent homology. It is for this reason that the barcode $B(T,f)$ resulting from the elder rule is known to remain the same after reordering leaves with identical $f$-values.
\end{remark}

\begin{definition}
    A \emph{partial matching} between barcodes $B$ and $B'$ is an injective map $\phi$ from a sub-multiset $S$ of $B$ into $B'$. Given an interval $I$ let $L(I)$ and $R(I)$ denote its (potentially infinite) left and right endpoints. The \emph{cost} of $\phi$ is the maximum of the following three values.
    \begin{align*}
        &\sup_{I\in S} \max\big(|L(I)-L(\phi(I))|, |R(I) - R(\phi(I))|\big),\\
        &\sup_{I \in B-S} \frac{R(I)- L(I)}{2}, \\
        &\sup_{I \in B' - \im\phi} \frac{R(I) -L(I)}{2}.
    \end{align*}
    In the above expressions any difference $a-b$ is assumed to be infinite if either $a$ or $b$ is infinite, unless $a=b$, in which case we take $a-b = 0$.

    The \emph{bottleneck distance} between barcodes $B$ and $B'$ is defined as the infimum of the cost of all partial matchings between $B$ and $B'$.
\end{definition}

If both $B$ and $B'$ are finite multisets, it can be shown that the infimum defining the bottleneck distance between $B$ and $B'$ is attained by some partial matching (see e.g. \cite{kerber2017geometry}). It is well known that the bottleneck distance is an extended pseudometric, i.e. it satisfies all of the axioms of a metric, except that the bottleneck distance between barcodes $B$ and $B'$ may be infinite, and may be equal to zero even if $B \neq B'$.

\begin{definition}
     We define the \emph{bottleneck distance} between merge trees $(T,f)$ and $(T',f')$, denoted $d_B\big((T,f),(T',f')\big)$, to be the bottleneck distance between $B(T,f)$ and $B(T',f')$, the barcodes obtained from $(T,f)$ and $(T',f')$ respectively via the elder rule.
\end{definition}

The following result relating the bottleneck and interleaving distances for merge trees comes from \cite[Theorem 3]{morozov2013interleaving}.

\begin{theorem}
    \label{thm:bottlelessthaninterleaving}
    On the space $\MT$, $d_B \leq d_I$.
\end{theorem}

From this and Lemma \ref{lem:intmetric} one may show that $d_B$ is a pseudometric on $\MT$, i.e. $d_B$ satisfies all the axioms of a metric except $d_B\big((T,f),(T',f')\big)$ may be equal to zero for $(T,f) \neq (T',f')$ (e.g. Figure \ref{fig:mergetreepath}). In particular $d_B$ never returns the value $\infty$ on $\MT$.

\subsection{Intrinsic Distances}

For any (pseudo)metric on a topological space, one can define another related (pseudo)metric via continuous path length:

\begin{definition}
    \label{def:intrinsic}
    Let $X$ be a topological space and $d:X\times X\to \mathbb{R}$ be a (pseudo)metric. If $\gamma:[0,1] \to X$ is a continuous path (with respect to the topology on $X$), we define the length of $\gamma$ with respect to $d$ as
    \begin{equation*}
        L_d(\gamma) := \sup_{\substack{n\\ 0=t_0\leq \ldots \leq t_n=1}} \sum_{i = 1}^n d\big(\gamma(t_i),\gamma(t_{i-1})\big).
    \end{equation*}
    Then we set
    \begin{equation*}
        \widehat{d}(x_0,x_1) = \inf_\gamma L_d(\gamma),
    \end{equation*}
    where the infimum is to be taken over all paths $\gamma$ that are continuous in the topology on $X$, start at $x_0$, and end at $x_1$. We call $\widehat{d}$ the \emph{intrinsic (pseudo)metric induced by $d$ on the space $X$.}

    A \emph{geodesic (of $d$)} between $a$ and $b$ in $X$ is a continuous path $\gamma:[0,1]\to X$ such that
    \begin{itemize}
        \item $\gamma(0) = a$;
        \item $\gamma(1) = b$; and
        \item $L_{d}(\gamma) = \widehat{d}(a,b)$.
    \end{itemize}
\end{definition}

We illustrate a path between two non-isomorphic merge trees with identical barcodes in Figure \ref{fig:mergetreepath}.

\begin{figure}[htbp]
    \centering
    \includegraphics[height=2.5in]{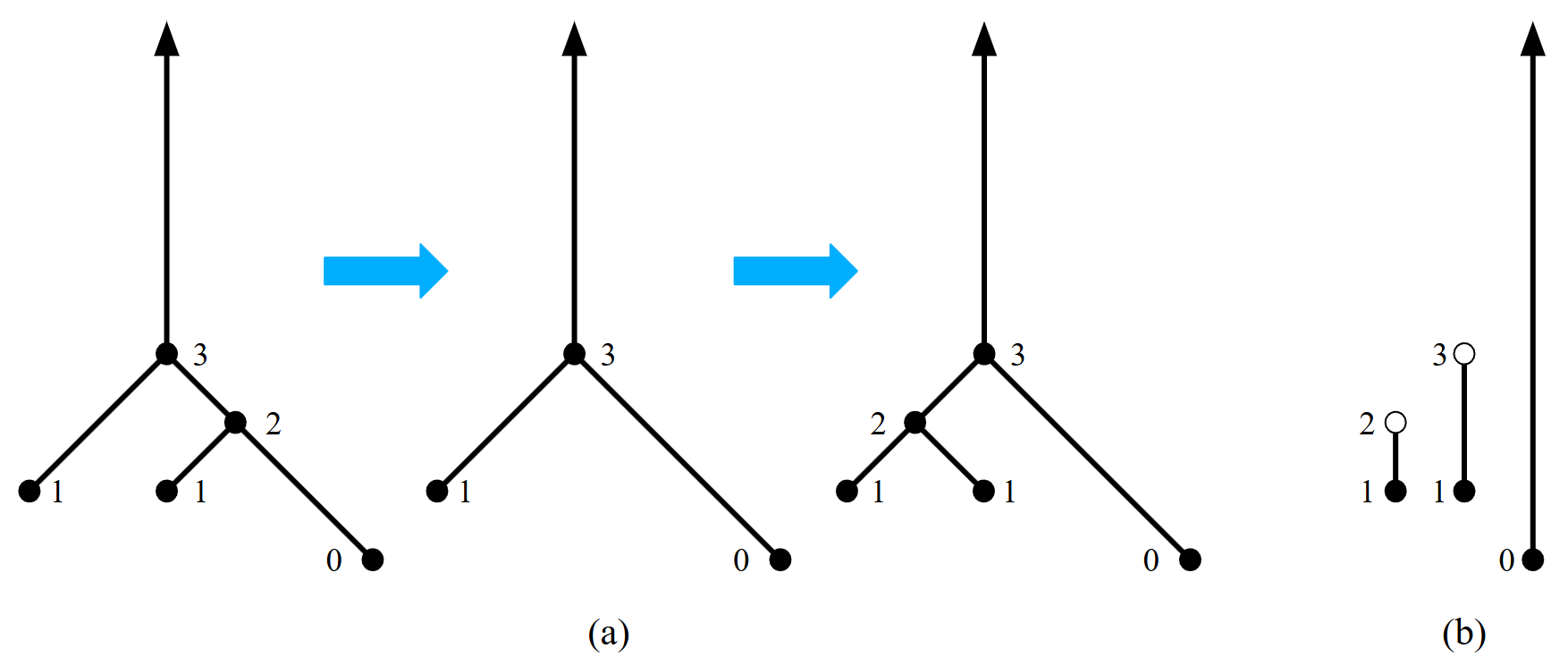}
    \caption{(a) The first, middle, and last points along a path $\gamma$ in $\MT$ between two merge trees with $d_B = 0$ but $d_I=1$. (b) The barcode of the merge trees at the endpoints in $\gamma$; the barcode at the midpoint is the same but missing the $[1,2)$ bar, which shows that $L_{d_B}(\gamma)\geq 1$. The first and last merge trees along the path cannot be isomorphic because the least common ancestors of the two leaves with $f$-value $1$ have different $f$-values.}
    \label{fig:mergetreepath}
\end{figure}

The main goal of this paper is to prove a conjecture from \cite{gasparovic2025intrinsic}, which says that on $\MT$ we have that $\widehat{d}_B = \widehat{d}_I$. However the definition above deviates from the definition of an intrinsic metric in \cite{gasparovic2025intrinsic} and so a few remarks are in order. In the definition of an intrinsic metric given in \cite{gasparovic2025intrinsic}, the authors assume $d$ is a metric and the infimum is taken over paths $\gamma$ that are continuous with respect to the topology given by $d$. Hence $\widehat{d}_I$ as defined here is identical to $\widehat{d}_I$ as defined in \cite{gasparovic2025intrinsic} (recall we equip $\MT$ with the topology induced by $d_I$).

We want to consider pseudometrics because $d_B$ is a pseudometric on $\MT$: there are non-isomorphic merge trees $(T,f)$ and $(T',f')$ with $d_B\big((T,f),(T',f')\big)=0$ (see Figure \ref{fig:mergetreepath} and \cite{curry2018fiber}). Thus paths in $\MT$ that are continuous in $d_B$ can `teleport' between the distinct merge trees $(T,f)$ and $(T',f')$. Hence, taking the definition of intrinsic distance as in \cite{gasparovic2025intrinsic}, the conjecture from the same paper is false for uninteresting reasons: if we equip $\MT$ with the topology induced by $d_B$, then $d_B\big((T,f),(T',f')\big)=0$, which implies $\widehat{d}_B\big((T,f),(T',f')\big)=0$ but $\widehat{d}_I\big((T,f),(T',f')\big) \geq d_I\big((T,f),(T',f')\big) \neq 0$. Our own definition avoids this technical issue and with it the aforementioned conjecture holds true. We remark that our definition of intrinsic distances is unoriginal, even in topological data analysis, with previous uses of essentially the same definition appearing in \cite{Carrire2017LocalEA} and \cite{vipond2020local} for similar reasons. In fact, the overall strategy of Sections \ref{sec:loceq} and \ref{sec:intrinicint} of this paper resembles that of Carri\`{e}re and Oudot \cite{Carrire2017LocalEA}, although in their work an inequality of intrinsic metrics is proven whereas we prove an equality.

A central result of \cite{gasparovic2025intrinsic} is the following.

\begin{theorem}\label{thm:dI}
    \label{prop:intrinsicinterleaving}
    On $\MT$, $d_I =\widehat{d}_I$.
\end{theorem}

\subsection{Good Maps and Labeled Merge Trees}

This section provides background necessary for Section \ref{sec:geodesics}, where we prove the existence of geodesics (with respect to $d_B$ and $d_I$) in $\MT$. The information provided here may be skipped by the reader who is only interested in earlier sections of the paper.

For convenience, we define $[n]$ to be the set $\{1, 2, \ldots,n\}$.

\begin{definition}
    A \emph{labeled merge tree} with $n$ labels is a triple $(T,f,\pi)$ where $(T,f)$ is a merge tree and $\pi:[n] \to T$ is a mapping containing each leaf of $T$ in its image. We denote by $\LMT_n$ the space of labeled merge trees with $n$ labels. We may associate to each labeled merge tree $(T,f,\pi)$ a $n \times n$ matrix $\mathcal{M}(T,f,\pi)$ given by $\mathcal{M}(T,f,\pi)_{ii} = f\pi(i)$ and $\mathcal{M}(T,f,\pi)_{ij} = f\big(\LCA(\pi(i),\pi(j))\big)$ when $i \neq j$.

    Given $(T,f,\pi), (T',f',\pi')\in \LMT_n$, the \emph{labeled interleaving distance} between the two labeled merge trees is defined as
    \begin{equation*}
        d_I^L\big((T,f,\pi), (T',f',\pi')\big) := \sup_{i,j}\big|\mathcal{M}(T,f,\pi)_{ij} - \mathcal{M}(T',f',\pi')_{ij}\big|\,.
    \end{equation*}
\end{definition}

By forgetting labelings, one may attain a merge tree from any labeled merge tree. Theorem 4.1 in \cite{gasparovic2025intrinsic} implies that, for any labeled merge trees $(T,f, \pi)$ and $(T',f',\pi')$,
\begin{equation*}
    d_I\big((T,f),(T',f')\big) \leq d_I^L\big((T,f,\pi), (T',f',\pi')\big)\,.
\end{equation*}

The other relevant definition is that of a $\delta$-good map.
\begin{definition}
    A \emph{$\delta$-good map} from a merge tree $(T,f)$ to another merge tree $(T',f')$ is a continuous map $\alpha:T\to T'$ such that
    \begin{enumerate}
        \item $f'\alpha(x)$ = $f(x) + \delta$ for all $x \in T$;
        \item if $w\in\im\alpha$ and $x' = \LCA(\alpha^{-1}(w))$, then $f'(x') - f(u) \leq 2\delta$ for all $u \in \alpha^{-1}(w)$; and
        \item For all $w \in T' -\im \alpha$, if $x' \preceq w$, then $f(w) - f(x) \leq 2\delta$.
    \end{enumerate}
\end{definition}

To compare merge trees, $\delta$-good maps are sometimes used in place of interleavings as it is known that the infimum of all $\delta$ such that there exists a $\delta$-good map from $(T,f)$ to $(T',f')$ is the interleaving distance between the two merge trees \cite{touli2022fpt}. We will use in particular that if $\alpha:T\to T'$ and $\beta:T' \to T$ form a $\delta$-interleaving between $(T,f)$ and $(T',f')$, then $\alpha$ is a $\delta$-good map from $(T,f)$ to $(T',f')$ \cite[Lemma 1]{touli2022fpt}.

\section{A Local Equality}
\label{sec:loceq}

The goal of this section is to prove the following theorem, from which the conjecture of Gasparovic et al. will follow without much more work.

\begin{theorem}
    \label{thm:localeq}
    Let $(T,f)\in\MT$. There is a neighborhood $W$ of $(T,f)$ in $\MT$ such that for all $(T',f')\in W$,
    \begin{equation*}
        d_B\big((T,f),(T',f')\big) = d_I\big((T,f),(T',f')\big)\,.
    \end{equation*}
\end{theorem}

Before we prove the theorem, we show that it would be impossible to strengthen this result to an equality between the interleaving and bottleneck distances on all pairs of merge trees in some neighborhood of $\MT$; that is, the local equality is genuinely a statement about distances \emph{to the fixed center} $(T,f)$, and cannot be upgraded to an isometry of a neighborhood.

\begin{proposition}\label{prop:no_isometry}
    Let $W$ be any neighborhood in $\MT$. Then there exist merge trees $(T_1,f_1), (T_2,f_2)\in W$ satisfying
    \begin{equation*}
        d_B\big((T_1,f_1), (T_2,f_2)\big) = 0 \qquad d_I\big((T_1,f_1), (T_2,f_2)\big) \neq 0\,.
    \end{equation*}
\end{proposition}

\begin{proof}
    It suffices to find merge trees $(T_1,f_1), (T_2,f_2)\in W$ that are not isomorphic but have identical barcodes. Fix a merge tree $(T,f)$ and $\epsilon_0>0$ such that the open ball $B_{\epsilon_0}(T,f)\subseteq W$. Let $C$ denote the (finite) set of $f$-values of leaves and branch points of $T$. Choose a point $x$ in the interior of some edge of $T$ with $f(x)\notin C$, and then choose $\epsilon>0$ small enough that
    \begin{equation*}
        \epsilon < \epsilon_0\,, \qquad \text{and} \qquad [f(x)-\epsilon,\ f(x)]\ \cap\ C = \varnothing\,.
    \end{equation*}
    Such an $\epsilon$ exists since $C$ is finite.

    Consider the barcode
    \begin{equation*}
        \big\{[f(x)-\epsilon, \infty),\ [f(x)-3\epsilon/4,\ f(x)),\ [f(x)-\epsilon/2,\ f(x)-\epsilon/4)\big\}\,,
    \end{equation*}
    shown in Figure \ref{fig:no_isometry}. Via the elder rule, this barcode is realized by two non-isomorphic merge trees, which we denote $(\tau_1,f_{\tau_1})$ and $(\tau_2,f_{\tau_2})$: the two finite bars may be attached to the infinite branch in either of two orders, and the resulting trees differ in which of the two shorter leaves shares a branch point with the third. These are the two non-isomorphic lifts of a three-bar barcode of the type described in \cite[Example 10]{curry2018fiber}; as $(\tau_1,f_{\tau_1})$ and $(\tau_2,f_{\tau_2})$ are non-isomorphic, $d_I\big((\tau_1,f_{\tau_1}),(\tau_2,f_{\tau_2})\big) = \delta$ for some positive $\delta$. On the other hand, $d_B((\tau_1,f_{\tau_1}),(\tau_2,f_{\tau_2})\big) = 0$ since the two trees have identical barcodes. We write $r_1, r_2$ for the roots of $\tau_1, \tau_2$, i.e. the unique points in $\tau_1$ and $\tau_2$ satisfying $f_{\tau_1}(r_1) = f_{\tau_2}(r_2) = f(x)$.

    We now graft $\tau_1$ onto $T$ at $x$. Let $\gamma_x \subseteq T$ be the increasing path from $x$ to $\infty$, and let $\gamma_{r_1} \subseteq \tau_1$ be the increasing path from the root $r_1$ to $\infty$. Define
    \begin{equation*}
        T_1 = (T \sqcup \tau_1)/\sim\,, \qquad \text{where } p \sim q \text{ if } p \in \gamma_x,\ q\in\gamma_{r_1}, \text{ and } f(p) = f_{\tau_1}(q)\,.
    \end{equation*}
    Since $f(x) = f_{\tau_1}(r_1)$, the identification glues $\gamma_x$ to $\gamma_{r_1}$ isometrically with respect to $f$-value, so
    \begin{equation*}
        f_1(p) := \begin{cases}
            f(p) & p \in T\,,\\
            f_{\tau_1}(p) & p \in \tau_1\,,
        \end{cases}
    \end{equation*}
    is well defined and continuous, and $(T_1,f_1)$ is a merge tree. Define $(T_2,f_2)$ analogously by grafting $\tau_2$ onto $T$ at $x$. We make the following observations.
    \begin{itemize}
        \item $i^\epsilon(T) = i^\epsilon(T_1) = i^\epsilon(T_2)$. Indeed, the only leaves and branch points of $T_1$ (resp. $T_2$) not already in $T$ are the non-root leaves and branch points of $\tau_1$ (resp. $\tau_2$), and all of these have $f$-value at most $f(x) - \epsilon/4 < f(x)$. Since $x \in \gamma_x = \gamma_{r_1}$, applying $i^\epsilon$ collapses the entire grafted subtree into the arc above $x$, recovering $i^\epsilon(T)$.
        \item Consequently $d_I\big((T,f),(T_j,f_j)\big) \leq \epsilon < \epsilon_0$, so $(T_1,f_1),(T_2,f_2)\in B_{\epsilon_0}(T,f) \subseteq W$. (The map $i^\epsilon$, together with the inclusion, exhibits an $\epsilon$-interleaving between each $(T_j,f_j)$ and $(T,f)$.)
        \item $d_B\big((T_1,f_1),(T_2,f_2)\big)=0$. Because the window $[f(x)-\epsilon, f(x)]$ meets no value of $C$, the three grafted bars form a disjoint addition to $B(T,f)$; as $\tau_1$ and $\tau_2$ contribute the same bars, $B(T_1,f_1) = B(T_2,f_2)$.
     \item $d_I\big((T_1,f_1), (T_2, f_2)\big) = \delta>0$. Suppose for contradiction that $\phi:T_1 \to T_2$ were a merge tree isomorphism. Then $\phi$ preserves $f$ and the ancestor relation, so it maps the height window $[f(x)-\epsilon, f(x)]$ to itself. By construction, this window contains only disjoint branch segments
     of $T$ and the branch containing $x$ has critical points coming from $\tau_1, \tau_2,$ respectively. As such, $\phi$ sends leaves in $T_1$ in the image of $\tau_1$ to leaves in $T_2$ in the image of $\tau_2$, while $\phi^{-1}$ sends leaves in $T_2$ in the image of $\tau_2$ to leaves in $T_1$ in the image of $\tau_1$. Since the mappings of $\tau_1$ into $T_1$ and of $\tau_2$ into $T_2$ are inclusions, $\phi$ and $\phi^{-1}$ hence restrict to a pair of continuous maps between $\tau_1$ and $\tau_2$ which are inverses to each other. This is impossible since $\tau_1$ and $\tau_2$ are not isomorphic, so $(T_1,f_1)$ and $(T_2,f_2)$ have positive interleaving distance.
    \end{itemize}
    Thus $(T_1,f_1), (T_2,f_2) \in W$ are non-isomorphic with identical barcodes, completing the proof.
\end{proof}

\begin{figure}
    \centering
    \includegraphics[width=0.265\linewidth]{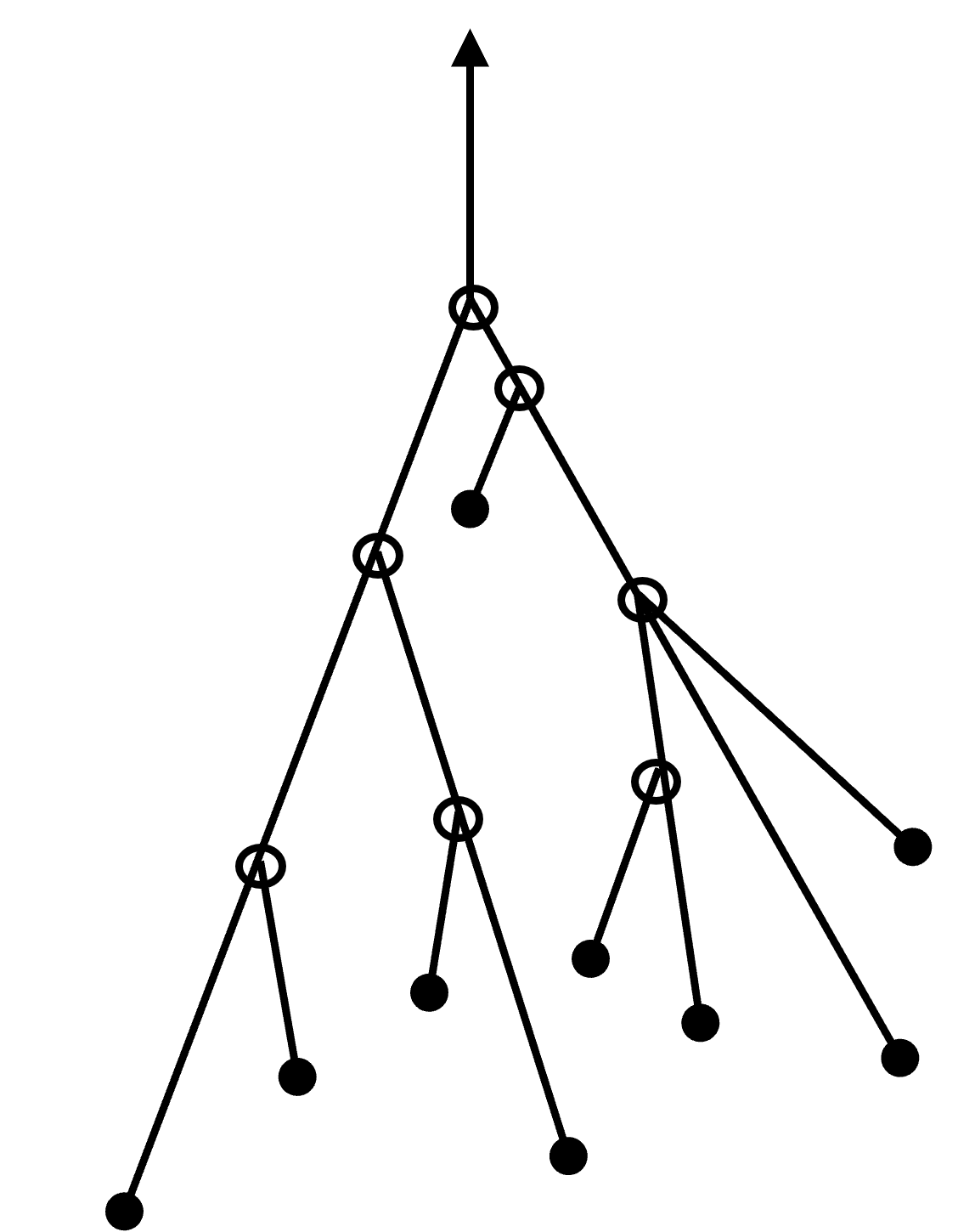}
    \includegraphics[width=0.62\linewidth]{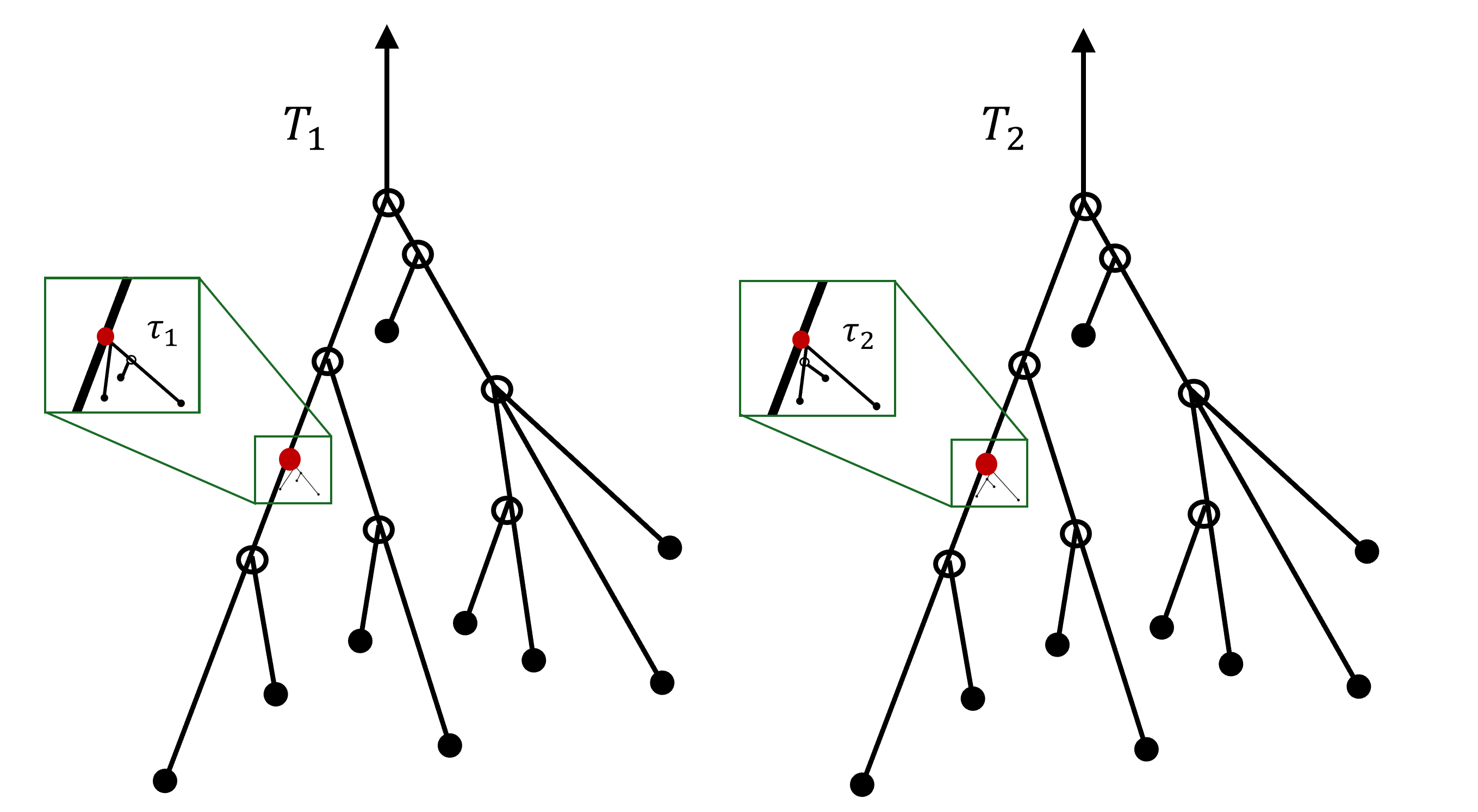}
    \includegraphics[width=0.9\linewidth]{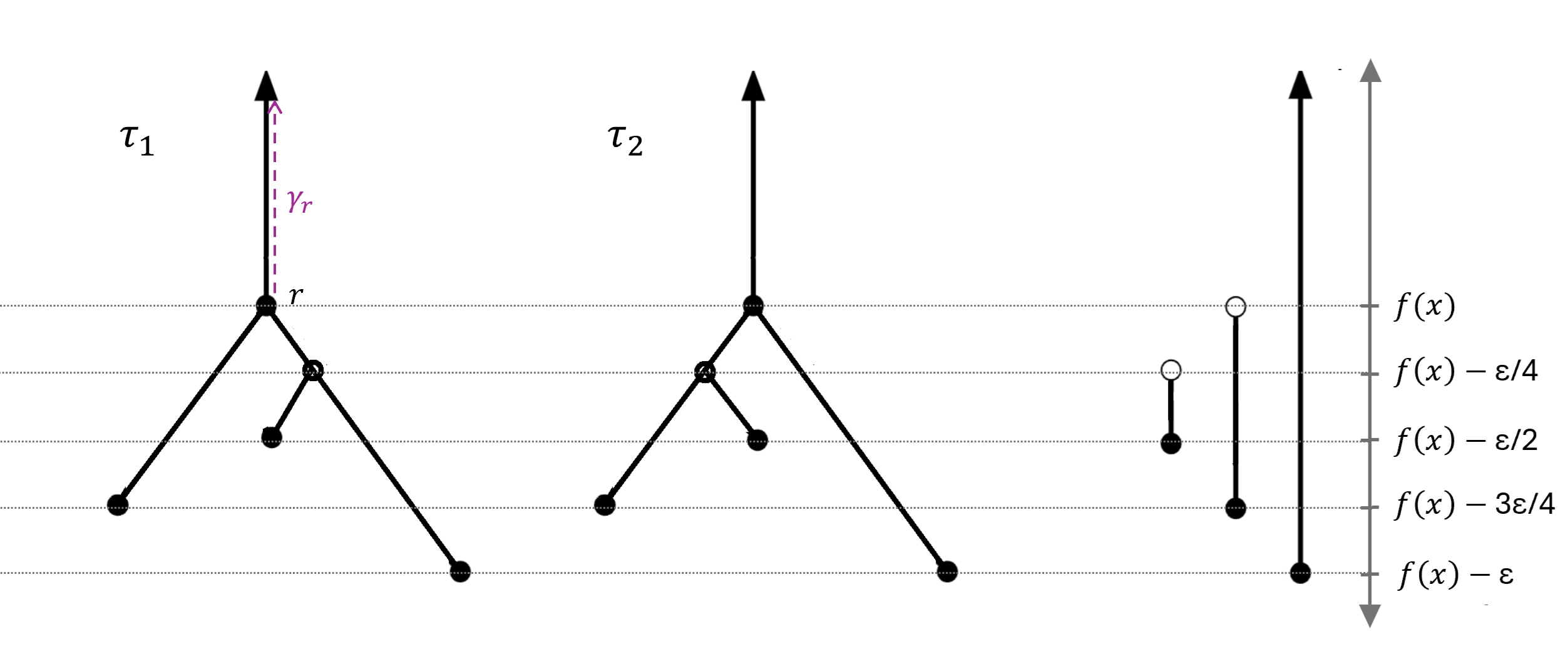}
    \caption{Construction in the proof of Proposition~\ref{prop:no_isometry}. Upper left, an example tree $T$ as in the proof. Below, $\tau_1$ and $\tau_2$ are constructed with function values indicated, so that $f_{\tau_1}(r_1) = f_{\tau_2}(r_2) = f(x)$, where $x\in T$ is the gluing point, marked in red; it is chosen so that the window $(f(x)-\epsilon, f(x))$ contains no leaf or branch point of $T$. Top center and right, $T_1$ and $T_2$ are formed by grafting $\tau_1$ and $\tau_2$ onto $T$, identifying $\gamma_{r_1}$ (resp. $\gamma_{r_2}$) with $\gamma_x$.}
    \label{fig:no_isometry}
\end{figure}

The rest of this section is devoted to proving Theorem \ref{thm:localeq}. We fix an arbitrary merge tree $(T,f)\in\MT$ throughout the remainder of this section and let $B$ denote the barcode attained by applying the elder rule to $(T,f)$. We let $C$ denote the set of $f$-values of leaves and branch points of $T$, noting that equivalently $C$ is the set of non-infinite endpoints of $B$, by the elder rule. We define $\epsilon$ to be a positive number sufficiently small that if $x,y \in C$ and $x\neq y$, then $|x-y| > 4\epsilon$.

We define $W$ to be the $\epsilon$-ball about $(T,f)$ in $\MT$, that is, the set
\begin{equation*}
    W = \big\{(T',f')\in \MT: d_I\big((T,f),(T',f')\big)<\epsilon\big\}\,.
\end{equation*}
We fix for the entire section a merge tree $(T',f')$ chosen arbitrarily from $W$ and an $\epsilon$-interleaving $\alpha:T \to T'$, $\beta:T'\to T$.

Let $l_0,\ldots,l_n$ be the leaves of $(T,f)$. We order these leaves such that if $f(l_i) < f(l_j)$, then $i < j$. We let $l'_0,\ldots,l'_m$ denote the leaves of $(T',f')$. We similarly order these leaves such that if $f'(l'_i) < f'(l'_j)$, then $i < j$.

\begin{lemma}
    \label{lem:sigma}
    For $0\leq i \leq n$, the optimization problem
    \begin{equation*}
        \argmin_{x\in T'} f'(x) \quad \text{subject to} \quad l_i \preceq \beta(x) \preceq i^{2\epsilon} (l_i)
    \end{equation*}
    has a solution, and any solution is a leaf. For each $0 \leq i \leq n$ we pick such a solution and denote it by $l'_{\sigma(i)}$. We have:
    \begin{equation*}
        |f'(l'_{\sigma(i)}) - f(l_i)| \leq \epsilon\,.
    \end{equation*}
    The function $\sigma$ is injective and satisfies that if $f(l_i) < f(l_j)$, then $f'(l'_{\sigma(i)}) < f'(l'_{\sigma(j)})$.
\end{lemma}

In particular, the fact that $\sigma$ is injective implies that $T$ has at most as many leaves as $T'$ has, i.e. $n \leq m$.

\begin{proof}
    First note that $\alpha(l_i)$ satisfies the constraint of the optimization problem, since $\beta \alpha (l_i) = i^{2\epsilon}(l_i)$, so we are not optimizing over the empty set. Let $I$ denote the set of points $y$ such that $l_i \preceq y \preceq i^{2\epsilon}(l_i)$. Then $I$ is compact (it is homeomorphic to a closed interval), and as interleaving maps are proper, $\beta^{-1}(I)$ is compact. Since $\beta^{-1}(I)$ is also nonempty, a minimum of $f'$ must be attained on $\beta^{-1}(I)$, proving a solution to this optimization problem exists.

    If $x$ is a solution to this optimization problem, suppose $x$ is not a leaf. Then there exists a leaf $l'_k \preceq x$ with $f'(l'_k) < f'(x)$. As such, $\beta(l'_k) \preceq \beta(x) \preceq i^{2\epsilon}(l_i)$, so it must be the case that $\beta(l'_k)$ is not an ancestor of $l_i$. This implies there is a branch point $b$ with $l_i \preceq b \preceq x \preceq i^{2\epsilon}(l_i)$, which is impossible as our choice of $\epsilon$ assures that there are no branch points in $f^{-1}(f(l_i),f(l_i) + 2\epsilon]$. So $x = l'_{\sigma(i)}$ is a leaf.

    Since $l_i \preceq \beta(l'_{\sigma(i)}) \preceq i^{2\epsilon}(l_i)$, we have $f(l_i) \leq f'\beta(l'_{\sigma(i)}) \leq f(l_i) + 2\epsilon$, which is equivalent to the statement
    \begin{equation*}
        |f'(l'_{\sigma(i)}) -f(l_i)| \leq \epsilon\,.
    \end{equation*}
    If $\sigma(i) = \sigma(j)$ for $i \neq j$, then in particular $\beta(l'_{\sigma(i)})$ is an ancestor of both $l_i$ and $l_j$. It follows that there is a branch point in $f^{-1}(f(l_i),f(l_i) + 2\epsilon]$, which is impossible. Hence $\sigma$ is injective.

    If $f(l_i) < f(l_j)$, then our choice of $\epsilon$ implies $f(l_j) - f(l_i) > 4\epsilon$. By the triangle inequality we conclude that $f'(l'_{\sigma(j)})-f'(l'_{\sigma(i)}) > 2\epsilon$, so that $f'(l'_{\sigma(i)}) < f'(l'_{\sigma(j)})$.
\end{proof}

We may reorder the leaves $l_i$ so that if $f(l_i) = f(l_j)$, but $f'(l'_{\sigma(i)}) < f'(l'_{\sigma(j)})$, then $i < j$. Moreover this reordering may be done while maintaining that $f(l_i) < f(l_j)$ implies $f'(l'_{\sigma(i)}) < f'(l'_{\sigma(j)})$. Once we have done this, we have that $\sigma$ is an increasing injective function.

Let $E = \{0,\ldots, m\} - \im\sigma$. We may reorder the leaves of $(T',f')$ so that if $f'(l'_i) = f'(l'_j)$ and $i \notin E$, but $j \in E$, then $i < j$. This reordering may be done while maintaining that $\sigma$ is increasing. Once we have done this, we are done reordering leaves of $(T,f)$ and $(T',f')$. In particular, our choice of ordering now implies that $\sigma(0) = 0$.

\begin{lemma}
    \label{lem:branchpoints}
    For each $1 \leq i \leq n$, there exists a unique solution to the optimization problem
    \begin{align*}
        \argmin_{x \in T}f(x) \quad \text{ subject to } \quad l_i, l_j \preceq x\text{, for some } j < i\,.
    \end{align*}
    For each $1 \leq i \leq m$, there exists a unique solution to the optimization problem
    \begin{equation*}
        \argmin_{x \in T'}f'(x) \quad \text{ subject to } \quad l'_i, l'_j \preceq x\text{, for some } j < i\,.
    \end{equation*}
    We denote the solutions to these optimization problems as $b_i$ and $b'_i$ respectively. These points are branch points and satisfy that for each $j \in E$,
    \begin{equation*}
        f'(b'_{j}) - f'(l'_j) \leq 2\epsilon\,.
    \end{equation*}
\end{lemma}

\begin{proof}
    Since $i > 0$, there exists $x$ satisfying $l_0,l_i \preceq x$, so the set we are optimizing over in the first problem is nonempty. We define
    \begin{align*}
        X &= \{x \in T: l_i \preceq x\}\,,\\
        Y &= \{x \in T: l_j \preceq x \text{ for some }j<i\}\,.
    \end{align*}
    Both sets are closed, so $X \cap Y$ is closed, and we have already verified $X\cap Y$ is nonempty. Since $X$ and $Y$ are bounded below in $f$-value, so is $X\cap Y$. It follows that a minimum of $f$ is attained on $X \cap Y$ as $f$ is proper and bounded below on $X \cap Y$, a nonempty set. Since $X \cap Y$ is the set we optimize over in the first optimization problem, this implies that the first optimization problem has a solution. This solution is unique since elements of $X$ (i.e. ancestors of $l_i$) are totally ordered by $f$. Hence $b_i$ is uniquely defined. The point $b_i$ must be a branch point as otherwise the greatest branch point descendant from $b_i$ also is a solution of the optimization problem. The proof that $b'_i$ is well defined, uniquely defined, and a branch point is analogous.

    For $j \in E$, there exists a leaf $l_i$ such that $l_i \preceq \beta(l'_j)$. It follows that:
    \begin{equation*}
        \beta(l'_{\sigma(i)}) \preceq \beta(l'_j) \implies \alpha\beta(l'_{\sigma(i)}) \preceq \alpha\beta(l'_j) \implies i^{2\epsilon}(l'_{\sigma(i)}) \preceq i^{2\epsilon}(l'_j) \implies l'_{\sigma(i)} \preceq i^{2\epsilon}(l'_j)\,.
    \end{equation*}
    Moreover, the fact that $\beta(l'_{\sigma(i)}) \preceq \beta(l'_j)$ implies that $f'(l'_{\sigma(i)}) \leq f'(l'_j)$, which implies that $\sigma(i) < j$. Hence, we conclude that
    \begin{equation*}
        l'_j \preceq b'_j \preceq i^{2\epsilon}(l'_j)\,,
    \end{equation*}
    implying that $f'(b'_j) - f'(l'_j) \leq 2\epsilon$, concluding the proof.
\end{proof}

At this point, we note that since $B$ is the barcode attained by applying the elder rule to $(T,f)$,
\begin{equation*}
    B = \big[f(l_0), \infty\big) \sqcup \bigsqcup_{i = 1}^n \big[f(l_i), f(b_i)\big)\,.
\end{equation*}

We let $B'$ denote the barcode obtained by applying the elder rule to $(T',f')$. Hence
\begin{equation*}
    B' = \big[f'(l'_0), \infty\big) \sqcup \bigsqcup_{i = 1}^m \big[f'(l'_i), f'(b'_i)\big)\,.
\end{equation*}

We let $\delta$ denote the cost of an optimal partial matching between $B$ and $B'$. It follows from Theorem~\ref{thm:bottlelessthaninterleaving} that $\delta \leq \epsilon$.

\begin{lemma}
    \label{lem:deltaclose}
    For $1 \leq i \leq n$,
    \begin{equation*}
        |f'(l'_{\sigma(i)}) - f(l_i)| \leq \delta\,,  \quad |f'(b'_{\sigma(i)}) - f(b_i)| \leq \delta\,,\quad \text{and}\quad f'(b'_{\sigma(i)}) - f'(l'_{\sigma(i)}) > 2\epsilon\,.
    \end{equation*}
    For $j \in E$,
    \begin{equation*}
        f'(b'_j) - f'(l'_j) \leq 2\delta\,.
    \end{equation*}
\end{lemma}

\begin{proof}
    Let $\phi:B-S \to B'$ be an optimal partial matching. Since, for $1\leq i \leq n$
    \begin{equation*}
        \frac{f(b_i) - f(l_i)}{2} \geq \frac{4\epsilon}{2} > \epsilon \geq \delta\,,
    \end{equation*}
    we conclude that $S$ must be empty, so that $\phi:B \to B'$. Since $f(b_i) - f(l_i) > 4\epsilon$ for $1 \leq i \leq n$ and $f'(b'_j) - f'(l'_j) \leq 2 \epsilon$ for all $j \in E$, we conclude for $1 \leq i \leq n$ and $j \in E$, that either
    \begin{equation*}
        |f'(b'_j) - f(b_i)| > \epsilon \geq \delta \quad \text{or} \quad |f'(l'_j) - f(l_i)| > \epsilon \geq \delta\,,
    \end{equation*}
    which implies that each interval $\big[f'(l'_j), f'(b'_j)\big)$ is not in the image of $\phi$. It follows that for each $j \in E$, we have
    \begin{equation*}
        \frac{f'(b'_j) - f'(l'_j)}{2} \leq \delta \iff f'(b'_j) - f'(l'_j) \leq 2\delta\,.
    \end{equation*}
    Moreover, it follows that $\phi$ surjects onto the submultiset of intervals
    \begin{equation*}
        \bigsqcup_{k = 1}^n\big[f'(l'_{\sigma(k)}), f'(b'_{\sigma(k)})\big)\,.
    \end{equation*}

    Hence we may choose $k$ such that $\phi$ maps $\big[f(l_k), f(b_k)\big)$ to $\big[f'(l'_{\sigma(i)}), f'(b'_{\sigma(i)})\big)$. If $f(l_i) \neq f(l_k)$, then $|f(l_k) - f(l_i)| > 4\epsilon$ due to our choice of $\epsilon$. The triangle inequality then would imply
    \begin{equation*}
        |f'(l'_{\sigma(i)}) - f(l_k)| > 3\epsilon > \delta\,,
    \end{equation*}
    contradicting that $\phi$ is a cost minimizing partial matching. Hence $f(l_i) = f(l_k)$ and so
    \begin{equation*}
        |f'(l'_{\sigma(i)}) - f(l_i)| =|f'(l'_{\sigma(i)}) - f(l_k)| \leq \delta\,.
    \end{equation*}

    Due to our choice of $\epsilon$, and the fact that $b'_k$ must be a branch point, we have
    \begin{equation*}
        f(b_k) - f(l_k) > 4\epsilon\,.
    \end{equation*}
    Since $\phi$ maps $\big[f(l_k), f(b_k)\big)$ to $\big[f'(l'_{\sigma(i)}), f'(b'_{\sigma(i)})\big)$, it follows that
    \begin{equation*}
        f'(b'_{\sigma(i)}) - f'(l'_{\sigma(i)}) \geq f(b_k) - f(l_k) - 2\delta > 4\epsilon - 2\delta \geq  2\epsilon\,.
    \end{equation*}
    Now let $1 \leq i \leq n$. By definition, there is an index $j< \sigma(i)$ such that $l'_j \preceq b'_{\sigma(i)}$. Since $f'(l'_j) \leq f'(l'_{\sigma(i)})$, it follows that $i^{2\epsilon}(l'_j) \preceq b'_{\sigma(i)}$. If $j \in E$, pick an index $r$ such that $l_r \preceq \beta(l'_j)$. It then follows that $\beta(l'_{\sigma(r)}) \preceq \beta(l'_j)$, which implies both that $l'_{\sigma(r)} \preceq i^{2\epsilon}(l'_{\sigma(r)}) \preceq i^{2\epsilon}(l'_{j}) \preceq b'_{\sigma(i)}$ and $\sigma(r) < j < \sigma(i)$. Otherwise, define $r$ such that $\sigma(r) = j$. Either way,
    \begin{align*}
        l'_{\sigma(r)}, l'_{\sigma(i)} \preceq b'_{\sigma(i)} \implies \beta(l'_{\sigma(r)}), \beta(l'_{\sigma(i)}) \preceq \beta(b'_{\sigma(i)}) &\implies l_r,l_i \preceq \beta(b'_{\sigma(i)})\\
        &\implies b_i \preceq \beta(b'_{\sigma(i)})\\
        &\implies f(b_i) - f'(b'_{\sigma(i)}) \leq \epsilon\,.
    \end{align*}
    For the other inequality, note that there exists $j<i$ such that $l_i,l_j \preceq b_i$. Since there are no leaves or branch points in $f^{-1}(f(l_i), f(l_i) + 2\epsilon]$ or $f^{-1}(f(l_j), f(l_j) + 2\epsilon]$, we have:
    \begin{align*}
        i^{2\epsilon}(l_i), i^{2\epsilon}(l_j) \preceq b_i \implies \beta(l'_{\sigma(i)}), \beta(l'_{\sigma(j)}) \preceq b_i &\implies \alpha\beta(l'_{\sigma(i)}), \alpha\beta(l'_{\sigma(j)}) \preceq \alpha(b_i)\\
        &\implies i^{2\epsilon}(l'_{\sigma(i)}), i^{2\epsilon}(l'_{\sigma(j)}) \preceq \alpha(b_i)\\
        &\implies l'_{\sigma(i)}, l'_{\sigma(j)} \preceq \alpha(b_i)\\
        &\implies b'_{\sigma(i)} \preceq \alpha(b_i)\\
        &\implies f'(b'_{\sigma(i)}) - f(b_i) \leq \epsilon\,.
    \end{align*}
    We have deduced that $|f'(b'_{\sigma(i)}) - f(b_i)| \leq \epsilon$. We conclude from this that $|f'(b'_{\sigma(i)}) - f(b_i)| \leq \delta$ in a manner similar to how we showed $|f'(l'_{\sigma(i)}) - f(l_i)| \leq \delta$.
\end{proof}

\begin{lemma}
    \label{lem:plemma}
    For each $j \in E$ there exists a unique solution to the optimization problem
    \begin{equation*}
        \argmin_{x \in T'} f'(x) \quad \text{subject to} \quad l'_{\sigma(i)}, l'_j \preceq x \text{, for some }\sigma(i) < j \,.
    \end{equation*}
    Denote this solution by $p'_j$. We have that each $p'_j$ is a branch point and satisfies
    \begin{equation*}
        f'(p'_j) - f'(l'_j) \leq 2\delta\,.
    \end{equation*}
    Let $q$ be the least index such that $l'_q \preceq p'_j$. Then $q \in \im \sigma$.
\end{lemma}

\begin{proof}
    The existence and uniqueness of $p'_j$ are proved similarly to the existence and uniqueness of the points $b_j$ and $b'_j$. The same is true for the fact that $p'_j$ is a branch point. For $j \in E$, we define $A(j)$ to be the least index such that $l'_{A(j)}\preceq b'_j$. We get a sequence $l'_j, l'_{A(j)}, \ldots, l'_{A^r(j)}, l'_{A^{r+1}(j)}$, where $j, A(j), \ldots, A^r(j) \in E$, but $A^{r+1}(j) \in \im \sigma$. From the definition of $A$ it follows that
    \begin{equation*}
        b'_j \preceq \ldots \preceq b'_{A^r(j)} \quad \text{and} \quad f'(l'_{A^r(j)})\leq \ldots \leq f'(l'_j)\,.
    \end{equation*}
    Moreover, since $l'_{A^{r+1}(j)} \preceq b'_{A^{r}(j)}$ and $l'_j \preceq b'_j \preceq b'_{A^{r}(j)}$, we deduce that $p'_j \preceq b'_{A^{r}(j)}$ since $A^{r+1}(j) \in \im \sigma$.
    Hence,
    \begin{equation*}
        f'(p'_j) - f'(l'_j) \leq f'(b'_{A^r(j)}) - f'(l'_{A^r(j)}) \leq 2\delta\,.
    \end{equation*}

    Now let $q$ be as in the statement of the lemma. If $q \in E$, then pick $k$ such that $l'_{\sigma(k)} \preceq p'_j$. From the definition of $q$ we deduce that $q < \sigma(k)$. It then follows that $l_k\preceq \beta(l'_{\sigma(k)}) \preceq \beta(p'_j)$. If $l_k \preceq \beta(p'_j) \preceq i^{4\epsilon}(l_k)$, then since there are no branch points in $f^{-1}(f(l_k), f(l_k) + 4\epsilon]$ and $\beta(l'_q) \preceq \beta(p'_j)$, it follows that $l_k \preceq \beta(l'_q)$. From the definition of $\sigma(k)$ it follows then that $\beta(l'_{\sigma(k)}) \preceq \beta(l'_q)$, from which we deduce that $f'(l'_{\sigma(k)}) \leq f'(l'_q)$. This implies $\sigma(k) < q$, a contradiction.

    Otherwise,
    \begin{align*}
        f\beta(p'_j) - f(l_k) \geq 4\epsilon \implies f'(p'_j) - f'(l'_{\sigma(k)}) \geq 2\epsilon &\implies f'(p'_j) - f'(l'_q) \geq 2\epsilon\\
        &\implies f'(p'_j) - f'(p'_q) \geq 0\,.
    \end{align*}
    Since both $p'_j$ and $p'_q$ are ancestors of $l'_q$, it follows that $p'_q \preceq p'_j$. Therefore there is an index $\sigma(d)$ with $\sigma(d) < q$ and $l'_{\sigma(d)} \preceq p'_q \preceq p'_j$, contradicting the definition of $q$. Hence $q$ is indeed in the image of $\sigma$.
\end{proof}

The final lemma we need before constructing a $\delta$-interleaving is the following.

\begin{lemma}
    \label{lem:lca}
    For $1 \leq i < j \leq n$,
    \begin{equation*}
        |f(\LCA(l_i,l_j)) - f'(\LCA(l'_{\sigma(i)}, l'_{\sigma(j)}))| \leq \delta\,.
    \end{equation*}
\end{lemma}

\begin{proof}
    For our convenience, we write $b = \LCA(l_i,l_j)$, $b' = \LCA(l'_{\sigma(i)}, l'_{\sigma(j)})$. Being least common ancestors, both $b$ and $b'$ are branch points. Since there are no branch points in $f^{-1}(f(l_i), f(l_i) + 2\epsilon]$, we have that $i^{2\epsilon}(l_i) \preceq b$. This implies
    \begin{equation*}
        l'_{\sigma(i)} \preceq i^{2\epsilon}(l'_{\sigma(i)}) = \alpha\beta(l'_{\sigma(i)}) \preceq \alpha i^{2\epsilon}(l_i) \preceq \alpha(b)\,.
    \end{equation*}
    Similarly, $l'_{\sigma(j)} \preceq \alpha(b)$. It follows that $b' \preceq \alpha(b)$, from which we deduce
    \begin{equation*}
        f'(b') - f(b) \leq \epsilon\,.
    \end{equation*}

    Again using that $i^{2\epsilon}(l_i) \preceq b$, we deduce $\beta(l'_{\sigma(i)}) \preceq i^{2\epsilon}(l_i) \preceq b$. Similarly, $\beta(l'_{\sigma(j)})\preceq i^{2\epsilon}(l_j) \preceq b$. We also have that $\beta(l'_{\sigma(i)}),\beta(l'_{\sigma(j)}) \preceq \beta(b')$. As such, we deduce
    \begin{equation*}
        b = \LCA(l_i,l_j) = \LCA(\beta(l'_{\sigma(i)}), \beta(l'_{\sigma(j)})) \preceq \beta(b')\,.
    \end{equation*}
    From this it follows that $f(b) - f'(b') \leq \epsilon$. So $|f'(b') - f(b)| \leq \epsilon$.

    Since $b$ is a branch point in $T$, $b = b_r$ for some index $r$. We wish to show $b' = b'_{\sigma(k)}$ for some index $k$. By picking a cellular structure on $T'$, define $e_i$ to be the edge incident to $b'$ such that $l'_{\sigma(i)} \preceq x \preceq b'$ for all $x \in e_i$. Define $e_j$ similarly. Pick $x_i \in e_i$ and $x_j \in e_j$. Let $\tau(i)$ denote the minimal integer such that $l'_{\tau(i)} \preceq x_i$ and define $\tau(j)$ similarly. If $\tau(i) \in E$, then we have $b' \preceq p'_{\tau(i)}$. Indeed, if this were not the case, then $p'_{\tau(i)} \preceq b'$ but $p'_{\tau(i)} \neq b'$, since both points are ancestors of $l'_{\tau(i)}$. Since $p'_{\tau(i)}$ is a branch point, this implies $p'_{\tau(i)} \preceq x_i$. Thus there exists some $k$ such that $\sigma(k) < \tau(i)$ and $l'_{\sigma(k)} \preceq p'_{\tau(i)} \preceq x_i$, contradicting the definition of $\tau(i)$. Thus $b' \preceq p'_{\tau(i)}$.

    Hence we have
    \begin{equation}
        \label{eqn:tauclose}
        f'(b') - f'(l'_{\tau(i)}) \leq f'(p'_{\tau(i)}) - f'(l'_{\tau(i)}) \leq 2\epsilon\,.
    \end{equation}
    On the other hand, since $\tau(i) \in  E$, we have that $\tau(i) \neq \sigma(i)$. Since $l'_{\tau(i)}, l'_{\sigma(i)} \preceq x_i$, it follows from the definition of $\tau(i)$ that $\tau(i) < \sigma(i)$. Hence, $b'_{\sigma(i)} \preceq x_i \preceq b'$. Therefore,
    \begin{equation*}
        f'(l'_{\tau(i)}) \leq f'(l'_{\sigma(i)}) \leq f'(b'_{\sigma(i)}) \leq f'(b')\,.
    \end{equation*}
    Hence,
    \begin{equation*}
        f'(b') - f'(l'_{\tau(i)}) \geq  f'(b'_{\sigma(i)}) - f'(l'_{\sigma(i)}) > 2\epsilon\,,
    \end{equation*}
    contradicting Equation (\ref{eqn:tauclose}). Therefore, $\tau(i) \notin E$, and so $\tau(i) = \sigma(q)$ for some index $q$. Similarly we may show $\tau(j) = \sigma(k)$ for some index $k$. Without loss of generality, we may assume that $\sigma(q) < \sigma(k)$. It then follows that $b' = b'_{\sigma(k)}$. Since $b = b_r$ we have
    \begin{equation*}
        |f(b_r) - f'(b'_{\sigma(k)})| = |f(b) - f'(b')| \leq \epsilon\,.
    \end{equation*}
    This implies
    \begin{equation*}
        |f(b_r) - f(b_k)| \leq 2\epsilon\,.
    \end{equation*}
    Due to our choice of $\epsilon$, and the fact that both $b_r$ and $b_k$ are branch points in $T$, this implies $f(b_r) = f(b_k)$. Hence,
    \begin{equation*}
        |f(b) - f'(b')| = |f(b_r) - f'(b'_{\sigma(k)})| = |f(b_k) - f'(b'_{\sigma(k)})| \leq \delta\,,
    \end{equation*}
    completing the proof.
\end{proof}

Now we can prove Theorem \ref{thm:localeq}.
\begin{proof}[Proof of Theorem \ref{thm:localeq}.] For $x = i^t(l_i)$, define
\begin{equation*}
    \widetilde{\alpha}_i(x) = i^{t + \delta +  f(l_i) - f'(l'_{\sigma(i)})}(l'_{\sigma(i)})\,.
\end{equation*}
This definition makes sense since $t + \delta +  f(l_i) - f'(l'_{\sigma(i)}) \geq t \geq 0$. Moreover, $f'\widetilde{\alpha}_i(x) = f(x) + \delta$.
If $x = i^t(l_i) = i^s(l_j)$, then $\LCA(l_i,l_j) \preceq x$. Therefore,
\begin{equation*}
    f'\widetilde{\alpha}_i(x) = f(x) + \delta \geq f(\LCA(l_i,l_j)) + \delta \geq f'(\LCA(l'_{\sigma(i)}, l'_{\sigma(j)}))\,.
\end{equation*}
Since $\LCA(l'_{\sigma(i)}, l'_{\sigma(j)})$ and $\widetilde{\alpha}_i(x)$ are both ancestors of $l'_{\sigma(i)}$, this means $\LCA(l'_{\sigma(i)}, l'_{\sigma(j)}) \preceq \widetilde{\alpha}_i(x)$. Similarly, $\LCA(l'_{\sigma(i)}, l'_{\sigma(j)}) \preceq \widetilde{\alpha}_j(x)$. Since $f'\widetilde{\alpha}_i(x) = f'\widetilde{\alpha}_j(x) = f(x) + \delta$, this implies $\widetilde{\alpha}_i(x) = \widetilde{\alpha}_j(x)$. Therefore the maps $\widetilde{\alpha}_i$ together define a map $\widetilde{\alpha}: T \to T'$. Since $f'\widetilde\alpha(x) = f(x) + \delta$ and $\widetilde\alpha(x) \preceq \widetilde\alpha(y)$ whenever $x \preceq y$, we deduce that $\widetilde\alpha$ is continuous.

Due to the last statement of Lemma \ref{lem:plemma}, for $k \in E$ we may define $w(k)$ such that $\sigma(w(k))$ is the least index of a leaf descendant from $p'_k$. With this in mind, for $x = i^t(l'_k)$, we define
\begin{equation*}
    \widetilde{\beta}_k(x) = \begin{cases}
        i^{t + \delta + f'(l'_k) - f(l_i)}(l_i)& k = \sigma(i)\,,\\
        i^{t + \delta + f'(l'_k) - f(l_{w(k)})}(l_{w(k)})& k \in E\,.
    \end{cases}
\end{equation*}

If $k = \sigma(i)$, this is well defined as $t + \delta + f'(l'_k) - f(l_i) \geq t \geq 0$, and moreover we have $f\widetilde{\beta}_k(x) = t + \delta + f'(l'_k) = f'(x) + \delta$.

If $k \in E$, $\widetilde\beta_k(x)$ is well defined as
\begin{align*}
    t + \delta + f'(l'_k) - f(l_{w(k)}) &= t + \delta + f'(l'_{\sigma(w(k))}) - f'(l'_{\sigma(w(k))}) + f'(l'_k) - f(l_{w(k)})\\
    &\geq t + \delta + f'(l'_{\sigma(w(k))}) - f(l_{w(k)})\\
    &\geq t \geq 0\,,
\end{align*}
where in going from the first to the second line we have used that since $l'_{\sigma(w(k))}, l'_k\preceq p'_k$, we have $f'(l'_{\sigma(w(k))}) \leq f'(l'_k)$, by the definition of $w$. We also have $f\widetilde\beta_k(x) = t + \delta + f'(l'_{k}) = f'(x) + \delta$.

We now wish to show that the maps $\widetilde\beta_k$ combine to define a map $\widetilde\beta:T' \to T$, as we did with $\widetilde\alpha$. Let $x = i^t(l'_k) = i^s(l'_q)$. There are three cases to handle.

The first case is that $k = \sigma(i)$ and $q = \sigma(j)$. In this case $\LCA(l'_{\sigma(i)}, l'_{\sigma(j)})\preceq x$, so
\begin{equation*}
    f\widetilde\beta_k(x) = f'(x) + \delta \geq f'(\LCA(l'_{\sigma(i)}, l'_{\sigma(j)})) + \delta \geq f(\LCA(l_i,l_j))\,.
\end{equation*}
Since both $\widetilde\beta_k(x)$ and $\LCA(l_i,l_j)$ are ancestors of $l_i$, it follows that $\LCA(l_i,l_j) \preceq \widetilde \beta_k(x)$. Similarly, $\LCA(l_i,l_j) \preceq \widetilde \beta_q(x)$. Since $f\widetilde\beta_k(x) = f\widetilde\beta_q(x) = f'(x) + \delta$, we conclude that $\widetilde\beta_k(x) = \widetilde\beta_q(x)$.

The second case is that $k = \sigma(i)$ and $q \in E$. If $x \preceq p'_q$, then it must be the case that $f'(l'_{\sigma(i)}) > f'(l'_q)$. Therefore, using that $f'(p'_q) - f'(l'_q) \leq 2\delta \leq 2\epsilon$, we have
\begin{align*}
    f'(l'_q) < f'(l'_{\sigma(i)}) \leq f'(x) \leq f'(p'_q) \implies f'(x) - f'(l'_{\sigma(i)}) \leq 2\epsilon \implies f\beta(x) - f(l_i) < 4\epsilon\,,
\end{align*}
so that $l_i \preceq \beta(l'_{\sigma(i)}) \preceq \beta(x) \preceq i^{4\epsilon}(l_i)$. Since there are no branch points in $f^{-1}(f(l_i),f(l_i) + 4\epsilon]$, we have that all descendants of $\beta(x)$ are ancestors of $l_i$. In particular, $l_i \preceq \beta(l'_q)$. This implies $\beta(l'_{\sigma(i)}) \preceq \beta(l'_q)$, and so $f'(l'_q) \geq f'(l'_{\sigma(i)})$, a contradiction.

It follows that $p'_q \preceq x$, so that $x = i^u(l'_{\sigma(w(q))})$. We have that $u = s + f'(l'_q) - f'(l'_{\sigma(w(q))})$, so
\begin{align*}
    \widetilde\beta_{\sigma(w(q))}(x) &= i^{u + \delta + f'(l'_{\sigma(w(q))}) - f(l_{w(q)})}(l_{w(q)})\\
    & = i^{s + \delta + f'(l'_q) - f'(l'_{\sigma(w(q))}) + f'(l'_{\sigma(w(q))}) - f(l_{w(q)})}(l_{w(q)})\\
    &= i^{s + \delta + f'(l'_q) - f(l_{w(q)})}(l_{w(q)})\\
    & = \widetilde\beta_{q}(x)\,.
\end{align*}
Hence, by the first case, $\widetilde\beta_k(x) = \widetilde\beta_{\sigma(w(q))}(x) = \widetilde{\beta}_{q}(x)$.

The final case is that $k, q \in E$. In this case, if $l'_{\sigma(i)} \preceq x$ for some $1 \leq i \leq n$, then $x = i^u(l'_{\sigma(i)})$ and so, using the second case, we have $\widetilde\beta_{k}(x) = \widetilde\beta_{\sigma(i)}(x) = \widetilde\beta_{q}(x)$. Hence we may assume there is no $1 \leq i \leq n$ such that $l'_{\sigma(i)} \preceq x$. As such, $x \preceq p'_k$, so $l'_q \preceq x \preceq p'_k$\,. From this we deduce that $p'_k = p'_q$. Hence $w(k) = w(q)$. Using this and that $t + f'(l'_k) = s + f'(l'_q)$, we have
\begin{align*}
    \widetilde{\beta}_k(x) &= i^{t + \delta + f'(l'_k) - f(l_{w(k)})}(l_{w(k)})\\
    &= i^{s + \delta + f'(l'_q) - f(l_{w(q)})}(l_{w(q)})\\
    & = \widetilde\beta_q(x)\,.
\end{align*}

Hence the $\widetilde\beta_k$ together define a map $\widetilde\beta: T' \to T$. Since $f\widetilde\beta(x) = f'(x) + \delta$ and $\widetilde\beta(x) \preceq \widetilde\beta(y)$ whenever $x \preceq y$, $\widetilde\beta$ is continuous.

It only remains to verify the identities $\widetilde\alpha\widetilde\beta = i^{2\delta}$ and $\widetilde\beta\widetilde\alpha = i^{2\delta}$. If $x = i^t(l_i)$, then
\begin{align*}
    \widetilde\beta\widetilde\alpha(x) = \widetilde\beta i^{t + \delta +  f(l_i) - f'(l'_{\sigma(i)})}(l'_{\sigma(i)}) = i^{t + \delta +  f(l_i) - f'(l'_{\sigma(i)}) + \delta + f'(l'_{\sigma(i)}) - f(l_i)} (l_i) = i^{2\delta} i^t(l_i) = i^{2\delta}(x)\,.
\end{align*}

If $x \in T'$, first assume $x = i^t(l'_{\sigma(i)})$, then
\begin{align*}
    \widetilde\alpha\widetilde\beta(x) = \widetilde\alpha i^{t + \delta +  f'(l'_{\sigma(i)}) - f(l_{i})}(l_i) = i^{t + \delta +  f'(l'_{\sigma(i)}) - f(l_{i}) + \delta + f(l_i) - f'(l'_{\sigma(i)})} (l'_{\sigma(i)}) = i^{2\delta} i^t(l'_{\sigma(i)}) = i^{2\delta}(x)\,.
\end{align*}
Otherwise, $x = i^t(l'_j)$, for some $j \in E$. In this case we have $f'\widetilde\alpha\widetilde\beta(x) = t + 2\delta + f'(l'_j) = f'i^{2\delta}(x)$. Since $f'(p'_j) - f'(l'_j) \leq 2\delta$ and $l'_j$ is a descendant of $p'_j$, it follows that $l'_{\sigma(w(j))} \preceq p'_j \preceq i^{2\delta}(l'_j) \preceq i^{t+2\delta}(l'_j) = i^{2\delta}(x)$. Since $\widetilde\alpha\widetilde\beta(x)$ is also an ancestor of $l'_{\sigma(w(j))}$ with the same $f$-value as $i^{2\delta}(x)$, it follows that $\widetilde\alpha\widetilde\beta(x) = i^{2\delta}(x)$.

Hence $\widetilde\alpha$ and $\widetilde\beta$ define a $\delta$-interleaving between $(T,f)$ and $(T',f')$, establishing that
\begin{equation*}
    d_B\big((T,f),(T',f')\big) \geq d_I\big((T,f),(T',f')\big)\,.
\end{equation*}
Then Theorem \ref{thm:bottlelessthaninterleaving} implies that in fact
\begin{equation*}
    d_B\big((T,f),(T',f')\big) = d_I\big((T,f),(T',f')\big)\,,
\end{equation*}
and the proof is complete.
\end{proof}

\section{Proving the Conjecture}
\label{sec:intrinicint}

The goal of this section is to strengthen Theorem \ref{prop:intrinsicinterleaving} to the following.

\begin{theorem}
    \label{thm:conjecture}
    On $\MT$, $\widehat{d}_B = \widehat{d}_I = d_I$.
\end{theorem}

Using Theorem \ref{thm:localeq}, we can prove this statement using a general metric spaces argument. In fact, we can actually prove something a little stronger.

\begin{corollary}
    \label{cor:pathssame}
    For any continuous path $\gamma:[0,1] \to \MT$, we have that $L_{d_I}(\gamma) = L_{d_B}(\gamma)$.
\end{corollary}

\begin{proof}
    For numbers $0= t_0 \leq \ldots \leq t_l = 1$ and $0\leq s_0\leq \ldots \leq s_m\leq 1$, we say $(s_0,\ldots,s_m)$ is subordinate to $(t_0,\ldots,t_l)$ if for each $0\leq i\leq l$ there is an index $j$ such that $s_j =t_i$.

    We claim that given any continuous path $\gamma:[0,1] \to \MT$, and any $0= t_0 \leq \ldots \leq t_l = 1$, there is a tuple $(s_0,\ldots,s_m)$ subordinate to $(t_0,\ldots,t_l)$ such that
\begin{equation}
\label{eqn:subordinatedist}
    \sum_{i=1}^md_B\big(\gamma(s_{i-1}),\gamma(s_i)\big) = \sum_{i=1}^m d_I\big(\gamma(s_{i-1}),\gamma(s_i)\big)\,.
\end{equation}
The claim and the triangle inequality together imply that $L_{d_B}(\gamma) = L_{d_I}(\gamma)$ for any continuous path $\gamma:[0,1] \to \MT$. To prove the claim, it suffices to prove the claim in the case $l = 1$.

By Theorem \ref{thm:localeq}, every $x\in [0,1]$ has a neighborhood $U_x \subseteq [0,1]$ such that for all $x' \in U_x$, $d_B\big(\gamma(x), \gamma(x')\big) = d_I\big(\gamma(x),\gamma(x')\big)$. By a compactness argument we may find $0 = t_0 = r_0 < \ldots< r_k = t_1 = 1$ such that $U_{r_{i-1}} \cap U_{r_{i}}$ is nonempty. Pick $r'_i \in U_{r_{i-1}} \cap U_{r_{i}}$ for $1\leq i \leq k$. Set $m = 2k$. For $0 \leq i \leq m$, define
\begin{equation*}
    s_i = \ \begin{cases}
        r_j & i = 2j\\
        r'_j &i = 2j - 1\,.
    \end{cases}
\end{equation*}
Then $d_B\big(\gamma(s_{i-1}), \gamma(s_i)\big) = d_I\big(\gamma(s_{i-1}),\gamma(s_i)\big)$ for $1 \leq i \leq m$. It follows that
\begin{equation*}
    \sum_{i=1}^md_B\big(\gamma(s_{i-1}),\gamma(s_i)\big) = \sum_{i=1}^m d_I\big(\gamma(s_{i-1}),\gamma(s_i)\big)\,,
\end{equation*}
proving the claim and hence the corollary.
\end{proof}

Theorem \ref{thm:conjecture} follows immediately.

\section{Existence of Geodesics}
\label{sec:geodesics}
In this section we establish the existence of geodesics (of $d_B$ and $d_I$) in the metric space $\MT$.

\begin{theorem}
    \label{thm:geodesics}
    Suppose $(T,f), (T',f') \in \MT$, such that $(T,f)$ has $n$ leaves and $(T',f')$ has $m$ leaves. Then there exists a geodesic from $(T,f)$ to $(T',f')$ consisting of merge trees with at most $n + m$ leaves. In other words, there exists a continuous map $\gamma:[0,1] \to \MT_{n + m}$ satisfying
    \begin{itemize}
        \item $\gamma(0) = (T,f)$;
        \item $\gamma(1) = (T',f')$; and
        \item $L_{d_B}(\gamma) = L_{d_I}(\gamma)  = d_I\big((T,f),(T',f')\big)$.
    \end{itemize}
\end{theorem}

We remark that from Theorem \ref{thm:dI} (which we remind the reader comes from \cite{gasparovic2025intrinsic}), the $d_I\big((T,f),(T',f')\big)$ in the last bullet point could be replaced with $\widehat{d}_I\big((T,f),(T',f')\big)$. From Theorem \ref{thm:conjecture}, it may be replaced by $\widehat{d}_B\big((T,f),(T',f')\big)$. Hence this theorem establishes the existence of geodesics of both the interleaving and bottleneck distances.

Our method of proof of Theorem \ref{thm:geodesics} is largely unoriginal. The reader will see that the proof below relies heavily on results and methods established in \cite{gasparovic2025intrinsic}.

\begin{proof}
    Suppose $d_I\big((T,f), (T',f')\big) = \delta$. By Lemma \ref{lem:infint}, there exists a $\delta$-interleaving given by $\alpha:T\to T'$ and $\beta:T' \to T$. By the proof of \cite[Lemma 1]{touli2022fpt}, the map $\alpha$ is a $\delta$-good map from $T$ to $T'$. By the proof of Theorem 4.1 in \cite{gasparovic2025intrinsic} (see in particular Lemmas 4.2-3 therein), we may equip $(T,f)$ with a labeling $\pi$ and $(T',f')$ with a labeling $\pi'$ such that both labelings are defined on $[n + m]$ and
    \begin{equation*}
        d_I^L\big((T,f,\pi), (T',f',\pi')\big) = \delta\,.
    \end{equation*}
    By \cite[Corollary 3.2]{gasparovic2025intrinsic} and its proof, there exists a path $\gamma'$ from $(T,f, \pi)$ to $(T',f',\pi')$ in $\LMT_{n+m}$ such that $L_{d_I^L}(\gamma') = \delta$. We may forget labelings to attain a path $\gamma:[0,1] \to \MT$ from $(T,f)$ to $(T',f')$. From \cite[Theorem 4.1]{gasparovic2025intrinsic} it follows that $L_{d_I}(\gamma) \leq L_{d_I^L}(\gamma')$. As such,
    \begin{equation*}
        L_{d_I}(\gamma) \leq L_{d_I^L}(\gamma') = \delta\,.
    \end{equation*}
    Also, $L_{d_I}(\gamma) \geq d_I\big((T,f),(T',f')\big) = \delta$, so $L_{d_I}(\gamma) = \delta = d_I\big((T,f), (T',f')\big)$. This proves $\gamma$ is a geodesic of $d_I$. By Corollary \ref{cor:pathssame}, $L_{d_B}(\gamma) = L_{d_I}(\gamma) = \delta$, so $\gamma$ is simultaneously a geodesic of $d_B$. Since $\gamma'$ is a path through the space of labeled merge trees with $n+m$ labels, each merge tree $\gamma'(t)$ must have at most $n + m$ leaves for $t \in [0,1]$. It then follows that $\gamma$ has image in $\MT_{n+m}$.
\end{proof}

\section{Discussion}
\subsection{Efficient computation of merge tree interleaving}
The barcode $B(T,f)$ of a merge tree $(T,f)$ can be computed in $O\big(k \log(k)\big)$ time, where $k$ is the number of leaves and branch points in $(T,f)$ (see e.g. \cite{edelsbrunner2010computational, rolle2024stable}). The bottleneck distance between $B(T,f)$ and $B(T',f')$ can be computed in $O\big(n^{1.5}\log(n)\big)$, where $n$ is the combined number of intervals in the barcodes or, equivalently, the number of leaves in $(T,f)$ and $(T',f')$ \cite{kerber2017geometry}. Consequently the bottleneck distance between merge trees can be efficiently computed. By contrast, computing the interleaving distance between merge trees is NP-hard; in fact, it is NP-hard to approximate within a factor of $3$ \cite{agarwal2018computing,touli2022fpt}. Our proof shows that in certain cases, the computation of $d_I\big((T,f),(T',f')\big)$ can be done in polynomial time, namely if $(T,f)$ is sufficiently close to $(T',f')$ we may compute the bottleneck distance instead.

\subsection{Future work and conjectures}
There are several natural open questions that are closely related to the work undertaken here, and in the spirit of \cite{gasparovic2025intrinsic} we state them here for others to consider.

First, we have already shown that if $(T,f)$ has $n$ leaves and $(T',f')$ has $m$ leaves, then there exists a geodesic (with respect to either the interleaving or bottleneck distance) through merge trees with $n + m$ or fewer leaves. We conjecture that we can only do slightly better than this:

\begin{conjecture}
    \label{conj:optgeo}
    Let $(T,f)$ and $(T',f')$ be merge trees with $n$ and $m$ leaves respectively. Then there exists a geodesic between $(T,f)$ and $(T',f')$ through merge trees with $n + m - 1$ or fewer leaves. In other words, there exists a continuous $\gamma:[0,1] \to \MT_{n + m -1}$ such that
    \begin{itemize}
        \item $\gamma(0) = (T,f)$;
        \item $\gamma(1) = (T',f')$; and
        \item $L_{d_B}(\gamma) = L_{d_I}(\gamma) = d_{I}\big((T,f),(T',f')\big)$.
    \end{itemize}
    Moreover, for all $n,m \geq 0$, there exist merge trees $(T,f)$ and $(T',f')$ with $n$ and $m$ leaves respectively such that no geodesic from $(T,f)$ to $(T',f')$ meets only merge trees with $n + m - 2$ or fewer leaves.
\end{conjecture}

We remark that the $n + m - 1$ in the above conjecture surely cannot be decreased by any constant. This can be seen by letting $(T,f)$ be a merge tree with one leaf and $(T',f')$ be a merge tree with $n$ leaves.

One may also wonder how intrinsic distances vary when we restrict to paths in $\MT_n$. We may define $\widehat{(d_B)}_n$ to be the intrinsic distance arising from the restriction of $d_B$ to $\MT_n$, and may define $\widehat{(d_I)}_n$ similarly for each $n\geq 1$. Corollary \ref{cor:pathssame} immediately implies $\widehat{(d_B)}_n = \widehat{(d_I)}_n$. Since $\MT_n \subseteq \MT$, we have $\widehat{(d_I)}_n \geq \widehat{d}_I = d_I$. We conjecture:
\begin{conjecture}
    When $n >1$, there exist merge trees $(T,f), (T',f')\in \MT_n$ such that
    \begin{equation*}
        \widehat{(d_I)}_n \big((T,f),(T',f')\big) > d_I \big((T,f),(T',f')\big)\,.
    \end{equation*}
\end{conjecture}
This conjecture is evidently closely related to Conjecture \ref{conj:optgeo}.

We also believe a similar result to our main theorem holds for the $p$-presentation distances on merge trees, defined in \cite{cardona2022universal}, and the $p$-bottleneck distances as defined there as well.

\section*{Acknowledgments}
We thank Jacob Leygonie for conversations which led us to our proofs of Lemmas \ref{lem:infint} and \ref{lem:intmetric}. We thank Riley Decker, Tung Lam, and H{\aa}vard Bjerkevik for conversations which helped us identify a mistake in a previous version of this manuscript, and the former two for further comments on our revision. We thank Vuka\v{s}in Stojisavljevi\'{c} and Lukas Waas for conversations which helped us develop Section \ref{sec:geodesics} and Conjecture \ref{conj:optgeo}. We thank Steve Oudot and Justin Curry for other helpful conversations. DB was supported by NSF RTG-2136090. GG was supported by EPSRC Centre to Centre Research Collaboration grant EP/Z531224/1 and NSF MSPRF-2202895.

\bibliographystyle{alpha}
\bibliography{refs}

\end{document}